\pdfoutput=1
\documentclass[reqno]{amsart}
\usepackage{bbm}
\numberwithin{equation}{section}
\usepackage{amssymb}

\usepackage[margin=1in]{geometry}

\usepackage{mathtools}
\mathtoolsset{showonlyrefs}

\usepackage{wasysym}
\usepackage[
  pdfauthor={Lo Chia-Chun},
  pdftitle={Homogenisation for the Robin eigenvalue problem on manifolds
    and flexibility of optimal Schrödinger potentials}]{hyperref}

\usepackage[normalem]{ulem}

\usepackage{accents}

\usepackage{graphicx}
\usepackage{tikz}

\usepackage{parskip}
\usepackage{enumerate}

\theoremstyle{definition}
\newtheorem{dfn}{Definition}[section]

\theoremstyle{plain}
\newtheorem{thm}[dfn]{Theorem}
\newtheorem{lem}[dfn]{Lemma}
\newtheorem{cor}[dfn]{Corollary}
\newtheorem{prop}[dfn]{Proposition}

\theoremstyle{remark}
\newtheorem{rmk}[dfn]{Remark}

\renewcommand{\emptyset}{\varnothing}

\newcommand{\Z}{\mathbb{Z}}
\newcommand{\N}{\mathbb{N}}
\newcommand{\Npos}{\mathbb{N}}

\newcommand{\R}{\mathbb{R}}
\newcommand{\Rpos}{\mathbb{R}_{>0}}
\newcommand{\Rnonneg}{\mathbb{R}_{\ge0}}
\newcommand{\C}{\mathbb{C}}

\renewcommand{\Re}{\operatorname{Re}}
\renewcommand{\Im}{\operatorname{Im}}

\newcommand{\at}[1]{\rvert_{#1}}

\newcommand{\inv}{^{-1}}

\newcommand{\dum}{\boldsymbol{\cdot}}

\newcommand{\card}[1]{\##1}

\newcommand{\disc}{D}

\newcommand{\sphere}[1][2]{S^{#1}}

\newcommand{\half}{\frac{1}{2}}
\newcommand{\thalf}{\tfrac{1}{2}}

\newcommand{\sign}{\operatorname{sign}}
\newcommand{\chf}[1]{\mathbbm{1}_{#1}}

\newcommand{\tendsto}{\to}

\newcommand{\dist}{\operatorname{dist}}
\newcommand{\supp}{\operatorname{supp}}

\newcommand{\eps}{\varepsilon}

\newcommand{\cl}[1]{\overline{#1}}
\newcommand{\interior}[1]{\operatorname{int}(#1)}
\newcommand{\bd}{\partial}

\newcommand{\vb}[1]{\mathbf{#1}}

\newcommand{\pd}[1]{\partial_{#1}}

\newcommand{\pdn}{\partial_{\vb{n}}}

\newcommand{\sequence}[2]{\{#1\}_{#2}}

\newcommand{\bigO}[1][]{O_{#1}}
\newcommand{\littleo}[1][]{o_{#1}}

\newcommand{\idot}[3][]{\langle#2,#3\rangle_{#1}}

\newcommand{\abs}[1]{\lvert#1\rvert}
\newcommand{\bigabs}[1]{\left\lvert#1\right\rvert}

\newcommand{\norm}[2][]{\lVert#2\rVert_{#1}}

\newcommand{\dual}[1]{#1^*}

\newcommand{\domain}[1]{\operatorname{Dom}(#1)}

\newcommand{\W}[2]{\mathrm{W}^{#1,#2}}
\newcommand{\Cn}[1]{\mathrm{C}^{#1}}
\newcommand{\Lp}[1]{\mathrm{L}^{#1}}

\newcommand{\Wprod}{\mathcal{W}}

\newcommand{\Youngf}{\Phi}

\newcommand{\LlogL}[2]{\Lp{#1}(\log\Lp{})^{#2}}
\newcommand{\LlogLone}{\mathrm{L}\log\mathrm{L}}
\newcommand{\expL}[1]{\exp\Lp{#1}}
\newcommand{\expLone}{\exp\mathrm{L}}

\newcommand{\Mfd}{X}
\newcommand{\mfd}{M}

\newcommand{\submfd}{\Upsilon}

\newcommand{\mfdim}{d}

\newcommand{\pullb}[1]{#1^*}

\newcommand{\dee}{\mathrm{d}}
\newcommand{\gee}{g}

\newcommand{\Tp}[2]{T_{#1}{#2}}

\newcommand{\Nb}[1]{N{#1}}
\newcommand{\Np}[2]{N_{#1}{#2}}

\newcommand{\Nchart}{\varphi}
\newcommand{\Nbnhd}{N^\eps}
\newcommand{\Npnhd}{N^\eps_x}
\newcommand{\ball}[3][\gee]{\mathcal{B}_{#1}(#2,#3)}

\newcommand{\ann}[4][\gee]{\mathcal{A}_{#1}(#2,#3,#4)}

\newcommand{\divg}[1][\gee]{\mathrm{div}}
\newcommand{\grad}[1][\gee]{\nabla_{#1}}
\newcommand{\lapl}[1][\gee]{\Delta_{#1}}

\newcommand{\Vol}[1][\mfdim]{\mathrm{Vol}_{#1}}
\newcommand{\Area}{\mathrm{Vol}_{\mfdim-1}}

\newcommand{\volume}[1][\gee]{\dee{v_{#1}}}
\newcommand{\area}[1][\gee]{\dee{A_{#1}}}
\newcommand{\dv}[1][\gee]{\,\dee{v_{#1}}}
\newcommand{\dA}[1][\gee]{\,\dee{A_{#1}}}

\newcommand{\Eball}[1]{B_{#1}}
\newcommand{\Eann}[2]{A_{#1,#2}}

\newcommand{\Rob}{R}
\newcommand{\Sch}{S}
\newcommand{\Rch}{{}}
\newcommand{\Dir}{D}

\newcommand{\lam}[2][k]{\lambda_{#1}^{#2}}

\newcommand{\laam}[1][k]{\bar{\lambda}_{#1}}

\newcommand{\Rayleigh}[2][\gee]{\mathcal{R}^{#2}_{#1}}

\newcommand{\site}{s}
\newcommand{\sites}{S^\eps}

\newcommand{\bsites}{S_\bd^\eps}
\newcommand{\psites}{S_\odot^\eps}
\newcommand{\osites}{S_\ocircle^\eps}

\newcommand{\bdom}{\dome_\bd}
\newcommand{\pdom}{\dome_\odot}
\newcommand{\odom}{\dome_\ocircle}

\newcommand{\zbd}{\Gamma}
\newcommand{\bdmfd}{{\bd\mfd\setminus\zbd}}
\newcommand{\Wz}[2]{W_{0,\zbd}^{#1,#2}}

\newcommand{\radius}{r_{\eps,\bdconst}}
\newcommand{\Radius}{R_\eps}

\newcommand{\rrate}{r}
\newcommand{\Rrate}{R}

\newcommand{\cell}{Y^\eps_\site}
\newcommand{\hole}{\ball{\site}{\radius}}
\newcommand{\holes}[1][\eps,\bdconst]{T^{#1}}
\newcommand{\pcell}{V^{\eps,\bdconst}_\site}

\newcommand{\chart}{\varphi_\site}
\newcommand{\scale}{\kappa}

\newcommand{\atrace}{\gamma_{\scale,p}}
\newcommand{\trace}{\gamma^{\eps,\bdconst}_{\scale,p}} 

\newcommand{\dom}{\Omega}
\newcommand{\dome}[1][\eps,\bdconst]{\Omega^{#1}}

\newcommand{\Vee}{V}
\newcommand{\Veee}[1][\eps]{V^{#1}}

\newcommand{\nue}[1][\eps]{\nu^{#1}}
\newcommand{\dnue}[1][\eps]{\,\dee\nu^{#1}}

\newcommand{\bdconst}{\alpha}
\newcommand{\bdparam}{\sigma}
\newcommand{\Bdparam}{\Sigma}
\newcommand{\bdparame}[1][\eps,\bdconst]{{\sigma^{#1}}}

\newcommand{\gext}[1][\eps,\bdconst]{J^{#1}}

\newcommand{\yu}{u}
\newcommand{\yuk}[2][k]{\yu^{#2}_{#1}}

\newcommand{\yue}[1][\eps,\bdconst]{u^{#1}}
\newcommand{\yuek}[1][\eps,\bdconst]{\yue[#1]_k}

\newcommand{\Yu}{U}
\newcommand{\Yue}[1][\eps,\bdconst]{U^{#1}}
\newcommand{\Yuek}[1][\eps,\bdconst]{\Yue[#1]_k}

\newcommand{\Yuu}{\bar{U}}
\newcommand{\Yuuk}[1][k]{\Yuu_{#1}}

\newcommand{\yuu}{u}

\newcommand{\Kay}[1][p]{K_{#1}}

\newcommand{\Psii}{\Psi^\eps_\site}

\newcommand{\mue}[1][\eps,\bdconst]{\mu^{#1}}

\newcommand{\mues}{\mu^{\eps,\bdconst}_\scale}

\newcommand{\Xies}{\overline{\xi}^{\eps}_\scale}
\newcommand{\Mues}{\overline{\mu}^{\eps,\bdconst}_\scale}

\newcommand{\schf}{\mathfrak{h}}
\newcommand{\schop}{H}
\newcommand{\resolvent}[2][\lambda]{R(#1;#2)}

\newcommand{\rate}{\beta}
\newcommand{\Mcap}{\mathcal{M}_{\mathrm{cap}}}

\newcommand{\wse}{w^\eps_\site}
\newcommand{\we}{w^\eps}

\newcommand{\hyp}[1]{(\textbf{#1})}
\newcommand{\hypone}{\hyp{$\text{H}\boldsymbol{\dom_1}$}}
\newcommand{\hyptwo}{\hyp{$\text{H}\boldsymbol{\dom_2}$}}
\newcommand{\hypthree}{\hyp{HV}}

\title[Homogenisation for the Robin eigenvalue problem on manifolds]
      {Homogenisation for the Robin eigenvalue problem on manifolds
        and flexibility of optimal Schr\"odinger potentials}
\author{Lo Chia-Chun}
\date{\today}
\begin{document}
\begin{abstract}
  We show that the spectrum of a Schr\"odinger eigenvalue problem posed
  on a closed Riemannian manifold $M$ with non-negative potential can be
  approached by that of Robin eigenvalue problems with constant positive
  boundary parameter posed on a sequence of domains in $\mfd$. We
  construct these Robin problems by means of a homogenisation procedure.
  We show a similar result for compact manifolds with non-empty
  boundary and sign-indefinite potential; in this case the Robin
  boundary parameter can be taken to be constant on each boundary
  component and to have constant magnitude.

  As an application, we prove a flexibility result for optimal
  Schr\"odinger potentials: for certain problems where it is known
  that there exists some potential $V$ which extremises some
  Schr\"odinger eigenvalue, we show that this extremal
  eigenvalue is also approached by the corresponding eigenvalues for a
  sequence of smooth potentials which remain bounded away from~$V$ in
  some dual Sobolev space.
\end{abstract}
\maketitle
\tableofcontents
\clearpage
\section{Introduction and main results}
\label{sec:introduction}
Let $(\Mfd,\gee)$ be a $\mfdim$-dimensional complete smooth Riemannian
manifold, and let $\mfd\subseteq\Mfd$ be a bounded domain in $\Mfd$
with possibly empty Lipschitz boundary. We call such a domain a
\emph{compact Riemannian manifold with Lipschitz boundary}, noting in
particular that any $\mfd$ which is a compact smooth Riemannian
manifold with boundary in the usual sense can be realised also as an
example of such a domain. The \emph{Robin eigenvalue problem}
on~$\mfd$ with boundary parameter $\bdparam:\bd\mfd\to\R$ consists in
finding all non-trivial $(\lam[]{\Rob}, \yu)$ satisfying
\begin{equation}
  \label{eqn:robin}
  \begin{cases}
    -\lapl\yu=\lam[]{\Rob}\yu&\text{on $\mfd$}\\
    \pdn\yu+\bdparam\yu=0&\text{on $\bd\mfd$}.
  \end{cases}
\end{equation}
Here $-\lapl$ is the Laplace--Beltrami operator with respect to the
metric $\gee$, and $\pdn$ denotes the outward normal derivative at the
boundary.
For all the examples we construct in the present work, we shall have
that $\bdparam\in\Lp\infty(\bd\mfd)$, which is sufficient for the
Robin spectrum to consist only of discrete eigenvalues which form a
non-decreasing sequence accumulating only at $+\infty$:
\begin{equation}
  \label{eqn:robin-eigenvalues}
  \lam[1]{\Rob}(\mfd,\bdparam)\le
  \lam[2]{\Rob}(\mfd,\bdparam)\le
  \dots
  \tendsto\infty.
\end{equation}

The \emph{Schr\"odinger eigenvalue problem} on $\mfd$ with potential
$\Vee:\mfd\to\R$ consists in finding $\lam[]{\Sch}$ such that the
equation
\begin{equation}
  \label{eqn:schroedinger-1}
  -\lapl\yu+\Vee\yu=\lam[]{\Sch}\yu
\end{equation}
on $\mfd$ admits a non-trivial solution for $\yu$. For self-adjoint
boundary conditions and under appropriate assumptions on $\Vee$, the
Schr\"odinger spectrum consists also of discrete eigenvalues:
\begin{equation}
  \label{eqn:schroedinger-eigenvalues}
  \lam[1]{\Sch}(\mfd,\Vee)\le
  \lam[2]{\Sch}(\mfd,\Vee)\le
  \dots
  \tendsto\infty.
\end{equation}
In the present work, we consider mixed Dirichlet/Robin conditions
on~$\bd\mfd$.
\subsection{Homogenisation for the Robin Problem}
The main result presented in this work is that, subject to some
integrability condition on the potential, the Schr\"odinger spectrum
on $\mfd$ a Riemannian manifold with Lipschitz boundary can be
approached by the Robin spectra of a family of domains
$\sequence{\dome[\eps]}{\eps>0}$ as $\eps$ tends to $0$, in which
$\dome[\eps]\subseteq\mfd$ for each~$\eps$. The integrability
condition that we shall impose on the potential $\Vee$ depends on the
dimension of $\mfd$. We employ the following terminology to avoid
repetition:
\begin{dfn}[Admissible potentials]
  \label{dfn:admissible}
  For $\mfdim\ge2$, a potential $\Vee$ on a $\mfdim$-dimensional
  Riemannian manifold $\mfd$ with Lipschitz boundary is said to be
  \emph{admissible} if $\Vee\in\Lp{\frac{\mfdim}{2}}(\mfd)$ when
  $\mfdim\ge 3$, or if $\Vee\in\LlogLone(\mfd)$ when $\mfdim=2$.
\end{dfn}
\begin{rmk}
  In practice, the results in this work are first proven for
  $\Vee\in\Cn1(\mfd)$; the statements in full generality are then
  obtained via a continuity result for Schr\"odinger eigenvalues.
\end{rmk}
\begin{rmk}
  Admissibility of $\Vee$ implies that for every $u,v\in\W12(\mfd)$,
  the expression $\int_\mfd\Vee uv\dv$, which appears in the weak
  formulation~\eqref{eqn:weak-schroedinger} of the Schr\"odinger
  problem, is finite. This in turn is sufficient to ensure that the
  Schr\"odinger spectrum consists of discrete eigenvalues. When
  $\mfdim=2$, that we have chosen to define admissible potentials to
  be those potentials in $\LlogLone(\mfd)$, rather than in
  $\Lp{\frac{\mfdim}{2}}$, is a consequence of dimension $2$ being a
  critical case of the relevant Sobolev embedding.
\end{rmk}
Definitions related to the Orlicz--Sobolev space $\LlogLone(\mfd)$ are
reviewed in Section~\ref{sec:notation}.

Let us first state the result in the case where $\mfd$ is a closed
Riemannian manifold, and where the potential $\Vee$ is
non-negative. In this case, it is possible moreover to construct the
Robin problems on $\dome[\eps]$ so that the boundary parameter
$\bdparame[\eps]$ is taken to be a constant.
\begin{thm}[Schr\"odinger as homogenised Robin; on a closed Riemannian manifold]
  \label{thm:main-nonneg}
  Fix a constant $\bdconst>0$.  Let $\mfd$ be a closed Riemannian
  manifold of dimension $\mfdim\ge2$. Let the potential $\Vee$ on
  $\mfd$ be admissible and non-negative.
  Then there exists a family $\sequence{\dome[\eps]}{\eps>0}$ of
  domains $\dome[\eps]\subseteq\mfd$ such that for all $k\in\Npos$ we
  have
  \begin{equation}
    \label{eqn:main-nonneg}
    \lam{\Rob}(\dome[\eps],\bdconst)\tendsto\lam{\Sch}(\mfd,\Vee)
  \end{equation}
  as $\eps\tendsto0$.  
\end{thm}
Each domain $\dome[\eps]$ is constructed by removing from $\mfd$ a
subset $\holes[\eps]$ which is the disjoint union of many small
equidistributed metric balls, with the holes made smaller and more
numerous as $\eps$ tends to zero. This construction of such
\emph{perforated domains} is an example of a process known as
\emph{homogenisation}; some context for this subject is given in
Section~\ref{ssec:background}.

More generally, we show that a similar approximation is possible in
the setting where $\mfd$ has nonempty boundary, and where the
potential for the Schr\"odinger problem may change sign.

When $\bd\mfd$ is nonempty, is necessary to supply boundary
conditions. For this work we choose to impose mixed boundary
conditions where we require solutions $\yu$ to have zero trace on some
subset of the boundary, and satisfy Robin boundary conditions
everywhere else. That is, for some $\zbd\subseteq\bd\mfd$, we require
\begin{equation}
  \label{eqn:schroedinger-2}
  \begin{cases}
    \yu=0&\text{on $\zbd$}\\
    \pdn\yu+\Bdparam\yu=0&\text{on $\mfd\setminus\zbd$.}
  \end{cases}
\end{equation}
with $\Bdparam\in\Lp\infty(\bd\mfd\setminus\zbd)$. Note that we
recover as particular cases Dirichlet boundary conditions on $\bd\mfd$
by choosing $\zbd=\bd\mfd$, Robin boundary conditions by choosing
$\zbd=\emptyset$, and Neumann boundary conditions by choosing
$\zbd=\emptyset$ with $\Bdparam$ identically zero.

Let $\Wz12(\mfd)$ be the closure in $\W12(\mfd)$ of those functions on
$\mfd$ smooth up to the boundary vanishing on $\zbd$. The weak
formulation of the Schr\"odinger problem on $\mfd$ with admissible
potential $\Vee$ and boundary conditions~\eqref{eqn:schroedinger-2}
consists in finding non-trivial $\yu\in\Wz12(\mfd)$ and
$\lam[]{\Sch}\in\R$ such that, for all $v\in\Wz12(\mfd)$,
\begin{equation}
  \label{eqn:weak-schroedinger}
  \int_{\mfd}\grad \yu\cdot\grad v\dv
  +
  \int_{\mfd}\yu v\Vee\dv
  +
  \int_{\bdmfd}\yu v\Bdparam\dA
  =
  \lam[]{\Sch}
  \int_{\mfd}\yu v\dv.
\end{equation}
Since the inclusion $\Wz12(\mfd)\subseteq\W12(\mfd)$ is continuous and
$\W12(\mfd)$ in turn embeds compactly into $\Lp2(\mfd)$, it follows
from the spectral theorem for self-adjoint operators with compact
resolvent that we have discrete eigenvalues
\begin{equation}
  \lam[1]{\Sch}(\mfd,\Vee,\Bdparam)
  \le
  \lam[2]{\Sch}(\mfd,\Vee,\Bdparam)
  \le
  \dots\tendsto\infty.
\end{equation}
In this setting, we again construct perforated domains
$\dome[\eps]\subseteq\mfd$. When forming the set $\holes[\eps]$, we
now take care to place the holes such that their closures fall
entirely within the interior of $\mfd$, so that in particular we have
that $\bd\dome[\eps]$ is the disjoint union
$\bd\mfd\cup\bd\holes[\eps]$. We then consider an eigenvalue problem
on $\dome[\eps]$ for the Laplacian with mixed boundary conditions: on
$\bd\mfd$, we impose the same boundary
conditions~\eqref{eqn:schroedinger-2} as we have specified for the
Schr\"odinger problem; on $\bd\holes[\eps]$, we impose Robin boundary
conditions with parameter
$\bdparame[\eps]\in\Lp\infty(\bd\holes[\eps])$ which is to be
constructed. That is, we have the following eigenvalue problem, where
we seek non-trivial $(\lam[]{\Rob},\yu)$ satisfying
\begin{equation}
  \label{eqn:robin-2}
  \begin{cases}
    -\lapl\yu=\lam[]{\Rob}\yu&\text{on $\dome[\eps]$}\\
    \yu=0&\text{on $\zbd$}\\
    \pdn\yu+\Bdparam\yu=0&\text{on $\bdmfd$}\\
    \pdn\yu+\bdparame[\eps]\yu=0&\text{on $\bd\holes[\eps]$.}\\
  \end{cases}
\end{equation}
The weak formulation for this problem consists in finding non-trivial
$\lam[]{\Rob}\in\R$ and $\yu\in\Wz12(\dome[\eps])$ such that, for any
$v\in\Wz12(\dome[\eps])$,
\begin{equation}
  \label{eqn:weak-robin}
  \int_{\dome[\eps]}\grad\yu\cdot\grad v\dv
  +
  \int_{\bd\holes[\eps]}\yu v\bdparame[\eps]\dA
  +
  \int_{\bdmfd}\yu v\Bdparam\dA
  =
  \lam[]{\Rob}
  \int_{\dome[\eps]}\yu v\dv,
\end{equation}
where $\Wz12(\dome[\eps])$ is defined analogously to
$\Wz12(\mfd)$. Under the stated assumptions on $\bdparame[\eps]$ and
$\Bdparam$, the spectrum of~\eqref{eqn:robin-2} again consists of
discrete eigenvalues.

Let us nevertheless continue to refer to the eigenvalues of
\eqref{eqn:robin-2} as \emph{Robin eigenvalues}, even though only when
$\zbd$ is empty do we actually have Robin boundary conditions on the
entirety of $\bd\dome[\eps]$. We denote by
$\lam{\Rob}({\dome[\eps]},\bdparame[\eps],\Bdparam)$ the $k$-th Robin
eigenvalue.

We are now prepared to state the main result in the setting of a
Riemannian manifold with nonempty Lipschitz boundary, which is that
the Schr\"odinger eigenvalues $\lam{\Sch}(\mfd,\Vee,\Bdparam)$ are
approached by the corresponding Robin eigenvalues
$\lam{\Rob}(\dome[\eps],\bdparame[\eps],\Bdparam)$. Moreover,
for any fixed constant $\bdconst>0$, we show that we can choose
$\bdparame[\eps]$ to be constant on each connected component of
$\bd\holes[\eps]$. Recall again from our earlier description that each
of these components is the boundary of a small metric ball we have
removed from $\mfd$; on each of these components, $\bdparame[\eps]$
takes value either $\bdconst$ or~$-\bdconst$ depending on the sign of
$\Vee$ nearby, in a manner that is made precise in
Section~\ref{ssec:homogenisation-construction}.
\begin{thm}[Schr\"odinger as homogenised Robin; on a Riemannian manifold with Lipschitz boundary]
  \label{thm:main}
  Fix a constant $\bdconst>0$. Let $\mfd$ be a Riemannian manifold
  with Lipschitz boundary, and let $\Vee$ be an admissible potential
  on $\mfd$. Let $\zbd\subseteq\bd\mfd$, and let
  $\Bdparam\in\Lp\infty(\bdmfd)$.
  
  Then there exists a family of domains
  $\sequence{\dome[\eps]}{\eps>0}$ given by
  $\dome[\eps]=\mfd\setminus\holes[\eps]$ where $\holes[\eps]$ is
  contained compactly within~$\mfd$, and there exists a locally
  constant
  ${\bdparame[\eps]:\bd\holes[\eps]\to\{-\bdconst,\bdconst\}}$ such
  that for each $k\in\Npos$ we have
  \begin{equation}
    \lam{\Rob}({\dome[\eps]},\bdparame[\eps],\Bdparam)\tendsto\lam{\Sch}(\mfd,\Vee,\Bdparam)
  \end{equation}
  as $\eps\tendsto 0$.
\end{thm}
We draw attention now to the fact that a considerable amount of
freedom is afforded to us in the choice of the parameter $\bdconst>0$.
For \emph{any} fixed $\bdconst$, the families
$\sequence{\dome[\eps]}{\eps>0}$ in the statements of
Theorems~\ref{thm:main-nonneg} and~\ref{thm:main} can be constructed
such that we have the convergence to the Schr\"odinger spectrum as
$\eps\tendsto0$. In fact, we show that this convergence is still
achieved even if $\bdconst$ is given as a function of $\eps$, provided
that it does not grow or decrease too quickly as $\eps\tendsto 0$;
this is an essential ingredient that makes possible the proof of the
flexibility results outlined in the following section.

\subsection{Flexibility of Schr\"odinger potentials}
Given an optimisation problem, one may be interested in the problem of
\emph{stability}, which is whether a candidate which comes close to
optimising a given functional must also be close to an optimiser in
some other sense. For example, for the problem of minimising the first
Dirichlet eigenvalue among domains in Euclidean space, one has a bound
on the Fraenkel asymmetry of a domain $\dom\subseteq\R^n$ in terms of
the difference $\lam[1]{\Dir}(\dom)-\lam[1]{\Dir}(B)$ between the
first Dirichlet eigenvalue of $\dom$ and that of a ball $B$ of the
same volume. With respect to the Fraenkel asymmetry, one also has
stability for the Szeg\H o--Weinberger inequality for the first
Neumann eigenvalue, and for the Brock--Weinstock inequality for the
first non-trivial Steklov eigenvalue normalised by
volume. See~\cite[Chapter 7]{b-p} for a survey of related results.

On the other hand, Bucur and Nahon~\cite{bucur-nahon} show for any
smooth conformal domains $\dom_1,\dom_2\subseteq\R^2$ that there
exists a domain $\dom$ homeomorphic to $\dom_1$ such that $\bd\dom$ is
contained in an arbitrarily small tubular neighbourhood of
$\bd\dom_1$, whereas the Steklov eigenvalues of $\dom$, normalised by
perimeter, are arbitrarily close instead to those of $\dom_2$. This
demonstrates in particular that the Weinstock inequality for Steklov
eigenvalues normalised by perimeter exhibits \emph{flexibility} with
respect to the Hausdorff distance between boundaries: domains which
come close to maximising an eigenvalue may nevertheless be in this
sense very different from the optimiser.

The Steklov problem is also an illustrative example concerning the
matter of identifying suitable spaces in which to consider the notion
of proximity between candidates, and therefore the notion of
stability, for an optimisation problem. It is seen, again
in~\cite{bucur-nahon}, that for a domain $\dom\subseteq\R^2$
conformally equivalent to the disc by $\varphi:\dom\to\disc$, the
Steklov problem on $\dom$ is isospectral to a weighted Steklov problem
on the disc, with the weight on $\bd\disc$ given by the push-forward
of the measure on $\bd\dom$ along $\varphi$. Identifying conformal
domains with the weights they define on $\bd\disc$ in this way, the
perimeter-normalised Steklov eigenvalues are stable with respect to
the norm in $\W{-\half}2(\bd\disc)$.

In the present work, we apply the results we have obtained concerning
homogenisation of the Robin problem to the problem of maximising or
minimising Schr\"odinger eigenvalues.
It is necessary to put some conditions on the potential to prevent the
problem from being ill-posed or having trivial solutions. For example,
the problem of optimising the first eigenvalue among non-negative
potentials in the class
\begin{equation}
  \mathcal{V}^\infty_{a,b}=\Big\{
  V\in\Lp\infty(\dom)
  \,:
  a\le V\le b
  \Big\}
  \quad\text{or}\quad
  \mathcal{V}^p_{a}=\Big\{
  V\in\Lp p(\dom)
  \,:
  \int_{\dom}V^p\,\dee x=a
  \Big\}
\end{equation}
on a bounded domain $\dom\subseteq\R^\mfdim$ with Dirichlet boundary
conditions is well-studied; we refer to~\cite[Chapter~8]{henrot} for a
survey of related results.

One might also more generally consider the Schr\"odinger problem with
a measure in place of the potential, by way of the weak formulation;
the way in which this is done is discussed further in
Section~\ref{sec:flexibility}. This perspective motivates the present
work in the following way:
for perforated domains $\dome[\eps]=\mfd\setminus\holes[\eps]$ and
boundary parameters $\bdparame[\eps]:\bd\holes[\eps]\to\R$ appearing
in the statements of Theorems \ref{thm:main-nonneg} and
Theorem~\ref{thm:main}, we consider a potential on $\mfd$ defined as
the push-forward of the measure $\bdparame[\eps]\area$ along the
inclusion $\bd\holes[\eps]\subseteq\mfd$, where $\area$ is the
Hausdorff measure on $\bd\holes[\eps]$ of dimension $\mfdim-1$.
We find then that the very same methods with which we show the
spectrum convergence of Robin eigenvalues give us also that the
Schr\"odinger eigenvalues for potentials $\bdparame[\eps]\area$
converge to $\lam{\Sch}(\mfd,\Vee,\Bdparam)$.
We show furthermore that the measures $\bdparame[\eps]\area$ remain
bounded away from $\Vee$ in $\dual{\W1p(\mfd)}$ for $p\le p_0$. In
particular, for the problem of extremising a Schr\"odinger eigenvalue
among a class of potentials $\mathcal{V}$ which is big enough to
contain the measures $\bdparame[\eps]\area$, if moreover the extremum
is attained by an admissible potential, then the optimal potential is
flexibile with respect to the norm in~$\dual{\W1p(\mfd)}$.

This result is further improved in two ways. First, we show that one
can in turn approximate the measures $\bdparame[\eps]\area$---which
are concentrated on a submanifold of co-dimension one---by measures
with densities given by smooth functions, so we in fact have the
aforementioned flexibility as soon as $\mathcal{V}$ contains these
smooth potentials.
Second, for any $1<p_0<\frac{\mfdim}{\mfdim-1}$, we exhibit an example
which is a witness for flexibility in \emph{precisely} those spaces
$\dual{\W1p(\mfd)}$ with $p\le p_0$: that is, for any such $p_0$, we
construct a family of potentials which approach $\Vee$ in
$\dual{\W1p(\mfd)}$ for any $p>p_0$, and yet is bounded away from
$\Vee$ in any $\dual{\W1p(\mfd)}$ with $p$ no greater than $p_0$.
\begin{thm}[Flexibility of Schr\"odinger potentials]
  \label{thm:flexibility}
  Let $\mfd$ be a Riemannian manifold with Lipschitz boundary of
  dimension $\mfdim\ge 2$, and let $\Vee$ be an admissible potential
  on $\mfd$.  Let $1<p_0<\frac{d}{d-1}$. Let $\zbd\subseteq\bd\mfd$,
  and let $\Bdparam\in\Lp\infty(\bdmfd)$.

  Then there exist a family of functions
  $\sequence{\Vee_\eps}{\eps>0}$ in $\Cn\infty(\mfd)$ such that
  \begin{itemize}
  \item
    For each $k\in\Npos$,
    \begin{equation}
      \lam{\Sch}(\mfd,\Vee_\eps,\Bdparam)\tendsto\lam{\Sch}(\mfd,\Vee,\Bdparam)
    \end{equation}
    as $\eps\tendsto 0$,
  \item For all $p>p_0$, $\Vee_\eps\tendsto\Vee$ in $\dual{\W1p(\mfd)}$, and
  \item For all $1\le p\le p_0$, There exists a constant $C_{\mfd,\Vee}>0$
    such that
    \begin{equation}
      \norm[\dual{\W1p(\mfd)}]{\Vee-\Vee_\eps}>C_{\mfd,\Vee}
    \end{equation}
    for all $\eps>0$.
  \end{itemize}
\end{thm}
\begin{rmk}
  One step in the proof of Theorems~\ref{thm:main-nonneg}
  and~\ref{thm:main} is to show that the Schr\"odinger eigenvalues are
  continuous with respect to admissible potentials. This is done in
  Proposition~\ref{prop:schroedinger-approx}. What is essentially the
  same proof also shows continuity of eigenvalues with respect to the
  norm in $\dual{\W1p(\mfd)}$ for any
  $p<\frac{\mfdim}{\mfdim-1}$. That we also have flexibility of
  Schr\"odinger eigenvalues in all such $\dual{\W1p(\mfd)}$, then, is
  in line with an observation made by Karpukhin, Nahon, Polterovich,
  and Stern~\cite[Remark 1.4]{knps}, which is that continuity of
  eigenvalues with respect to some topology is often accompanied by
  the flexibility of the optimisation problem with respect to the same
  topology.
\end{rmk}
\subsection{Background in homogenisation methods}
\label{ssec:background}
We close the introductory section by giving a brief account of some
context surrounding the methods employed in the present work.

Cioranescu and Murat~\cite{cioranescu-murat} considered the Dirichlet
problem
\begin{equation}
  \label{eqn:dirichlet}
  \begin{cases}
    -\lapl[]\yu^\eps=f&\text{on $\dome[\eps]$}\\
    \yu^\eps=0&\text{on $\bd\dome[\eps]$}.
  \end{cases}
\end{equation}
with $f\in\Lp2(\dome[\eps])$, posed on
a family $\sequence{\dom^\eps}{\eps>0}$ of domains in $\R^\mfdim$
which are constructed by taking a bounded domain
$\dom\subseteq\R^\mfdim$ and removing from it a small ball of radius
$r_\eps>0$ around each point in the scaled integer lattice
$\eps\Z^\mfdim$.

Let $\Yu^\eps\in\W12_0(\dom)$ be the extension of $\yu^\eps$ by zero over the
holes. The behaviour of $\Yu^\eps$ as $\eps\tendsto0$ depends on the rate
at which the radii $r_\eps$ of the holes decrease with $\eps$. A
critical rate is identified, depending on the dimension $\mfdim$, such
that $\Yu^\eps$ converge weakly in $\W12_0(\dom)$ to $\yu\in\W12_0(\dom)$
which solves not the Dirichlet problem on $\dom$, but the problem for
the Laplacian plus a multiplication operator $\mu^D$:
\begin{equation}
  \label{eqn:strange-term}
  \begin{cases}
    (-\lapl[]+\mu^D)\yu=f&\text{on $\dom$}\\
    \yu=0&\text{on $\bd\dom$}.
  \end{cases}
\end{equation}
Kaizu~\cite{kaizu} presented analogous results for Robin and Neumann
boundary conditions. These results were strengthened by
Cherednichenko, Dondl, and R\"osler \cite{cdr}, who proved them for
unbounded domains, and showed convergence in the norm-resolvent sense.

Results of this kind suggest that one might hope to uncover new
results about one spectral problem by exhibiting it as the homogenised
limit of another. This strategy has been successfully applied, for
example, by Girouard, Karpukhin, and Lagac\'e \cite{gkl}, to obtain
optimal upper bounds for Steklov eigenvalues among planar domains of
fixed perimeter. This was done by showing that the homogenised limit
of the Steklov problem on a family of perforated domains is a weighted
Neumann problem, whose eigenvalues are in turn related to those of the
Laplacian on the sphere.

A homogenisation construction in the Riemannian setting is described
by Girouard and Lagac\'e \cite{girouard-lagace}. The novel feature
that makes this generalisation possible is a method for distributing
the holes on the manifold without depending on any periodic structure
of the manifold or of some ambient space, as such a structure is in
general unavailable. Here, the role of progressively fine lattices in
the Euclidean setting is played instead by a sequence of maximally
$\eps$-separated sets with decreasing $\eps$. The weighted Laplace
eigenvalue problem on a closed manifold is shown in this way to be the
homogenised limit of the Steklov problem on a family of perforated
domains. In addition to proving several eigenvalue inequalities, this
homogenisation result also has applications concerning free-boundary
minimal surfaces in the sphere, via results of Fraser and Schoen~\cite{fraser-schoen}.

\subsection{Plan of the paper}
\label{ssec:outline}
The remainder of this article is organised as follows:

In Section~\ref{sec:notation}, we fix notational conventions, and
briefly review a number of results from functional analysis which are
relevant to the proof of the main results.

We begin Section~\ref{sec:robin-homogenisation} by describing in
detail the construction of Robin problems on a family of perforated
domains, and then proceed to prove a number of propositions related to
analysis on these domains. In particular, towards the end of
Section~\ref{ssec:analytic-properties}, we prove
Proposition~\ref{prop:measures-converge}, which makes precise the
condition~(ii) described in the previous section. Using these results,
we then finish the proofs of the main results
Theorems~\ref{thm:main-nonneg} and~\ref{thm:main} in
Section~\ref{ssec:problem-in-limit}.

In the final section, we discuss the problem of optimisation of
Schr\"odinger eigenvalues, and prove Theorem~\ref{thm:flexibility}.

\section{Notation and conventions}
\label{sec:notation}
We summarise below some notation and terminology used throughout this
work, and recall some relevant results.

\subsection{Measures on Riemannian manifolds}
Henceforth we take \emph{Riemannian manifold} to mean a \emph{compact
Riemannian manifold with Lipschitz boundary}, which we recall is
defined to be a bounded domain with (possibly empty) Lipschitz
boundary in a complete smooth Riemannian manifold. This is following
the terminology of Karpukhin and Lagac\'e~\cite{karpukhin-lagace}.
We refer to Riemannian manifolds with empty boundary as \emph{closed}.

On a Riemannian manifold $\mfd$, we abuse notation and write $\volume$
for both the Riemannian volume measure and for its density: for
example, we write $\volume(U)=\int_U1\dv$ for the volume of
$U\subseteq\mfd$. Likewise, we write $\area$ for the Hausdorff measure
of co-dimension one restricted to a submanifold.

Where a measure on a Riemannian manifold is omitted in the notation,
it is taken to be the volume measure induced by the metric. For
example, the space $\Lp p(\mfd)$ on a Riemannian manifold $(\mfd,g)$
is the space $\Lp p(\mfd,\volume)$.

\subsection{Asymptotic notation}
In computations, we reserve the notation $C$ for a strictly positive
constant, which may vary line-by-line. Subscripts on the notation
indicate dependence on parameters; for example, we write $C_\alpha$
for a positive constant that can be chosen depending only on some
parameter $\alpha$.

We make extensive use of \emph{asymptotic notation}, summarised as
follows:
\begin{itemize}
\item
  We write \emph{$f(x)\lesssim h(x)$ as $x\tendsto a$}, or
  equivalently \emph{$f(x)=\bigO(h(x))$ as $x\tendsto a$}, to mean
  that there exists a constant $C>0$ such that $\abs{f(x)}\le
  C\abs{h(x)}$ for all $x$ sufficiently close to $a$.
\item
  We write \emph{$f(x)=\littleo(h(x))$ as $x\tendsto a$} to mean that
  $\lim_{x\tendsto a}\frac{f}{h}=0$.
\item
  We may omit the specification of the limit along which the
  asymptotics are considered, when it is is clear from context. For
  example, we may simply write $f\lesssim h$, $f=\bigO(h)$, or
  $f=\littleo(h)$, without explicitly stating that these are
  e.g. asymptotics as some argument $x$ tends to $a$.
\item
  We write $f\asymp h$ to mean that $f\lesssim h$ and $h\lesssim f$.
\item
  We also use the big-$O$ notation in place of explicitly writing down
  a term with the corresponding asymptotic behaviour. For example, if
  $f-1\lesssim h$, we may write $f=1+\bigO(h)$.
\item
  Subscripts on the notation indicate that the relevant constants may
  depend on the parameter in the subscript. This includes those that
  occur in the definition of the limit: for example, we may write
  \emph{$f(x,y)\lesssim_y h(x,y)$ as $x\tendsto 0$} to mean that for
  any $\eps>0$, there exists $\delta(y)>0$ such that $\abs{f(x,y)}\le
  C_y\abs{h(x,y)}$ for all $\abs{x}<\delta(y)$.
\end{itemize}

\subsection{Function spaces and embedding theorems}
\label{ssec:function-spaces-embedding-theorems}
Finally, let us recall some results concerning Sobolev spaces.  Here
we reproduce only the statements, and direct the reader to the
references for details:
Propositions~\ref{prop:sobolev-multiplication-i} and
\ref{prop:sobolev-multiplication-ii} appear as~\cite[Lemma~4.9]{gkl}
and \cite[Lemma~4.10]{gkl}, respectively. Refer also to \cite{mazya}
for results of the kind of \ref{prop:sobolev-multiplication-i}. For a
treatment of Orlicz spaces, including
Propositions~\ref{prop:expLone-dual-LlogLone}
through~\ref{prop:orlicz-sobolev-embedding}, see \cite{b-s}.
\begin{prop}[Sobolev multiplication theorem, I]
  \label{prop:sobolev-multiplication-i}
  Let $\mfd$ be a Riemannian manifold of dimension
  $\mfdim\ge3$. Multiplication of smooth functions extends by density
  to a continuous bilinear map
  $\W12(\mfd)\times\W12(\mfd)\to\W1{\frac{\mfdim}{\mfdim-1}}(\mfd)$.
\end{prop}
For $\mfdim=2$, we still have that multiplication is a continuous map
$\W12(\mfd)\times\W12(\mfd)\to\W1p(\mfd)$ for all $1\le p<2$, but not
for $p=2$. This is sharpened into the following:
\begin{prop}[Sobolev multiplication theorem, II]
  \label{prop:sobolev-multiplication-ii}
  Let $\mfd$ be a $2$-dimensional Riemannian manifold. Multiplication
  of smooth functions extends by density to a continuous bilinear map
  $\W12(\mfd)\times\W12(\mfd)\to\W1{2,-\half}(\mfd)$.
\end{prop}
The space $\W1{2,-\half}$ is an \emph{Orlicz--Sobolev space}. We
recall the relevant definitions:

For a Young function $\Youngf:\R\to\Rpos$ (that is, a convex even
function $\R\to\Rpos$ such that $\Youngf(0)=0$ and
$\lim_{s\to\infty}\Youngf(s)=\infty$), the \emph{Orlicz space}
$\Lp{\Youngf}(\mfd)$ consists of all the measurable functions such
that the integral
\begin{equation}
  \int_\mfd\Youngf\left(\frac{\abs{f}}{\eta}\right)\dv<\infty
\end{equation}
for some $\eta>0$, identifying all functions that are almost
everywhere the same. The space $\Lp{\Youngf}(\mfd)$ can be made into a
Banach space by endowing it with the \emph{Luxemburg norm}, defined
\begin{equation}
  \norm[\Lp{\Youngf}(\mfd)]{f}
  =
  \inf
  \left\{
  \eta>0
  \,:
  \int_\mfd\Youngf\left(\frac{\abs{f}}{\eta}\right)\dv\le 1
  \right\}.
  \end{equation}
For example, the $\Lp p$ space for $1\le p<\infty$ is an example of an
Orlicz space, with $\Youngf(s)=s^p$. Also of interest to us are the
following cases: where $\Youngf(s)=(s(\log s)^a)^p$, in which we
denote the associated Orlicz space $\LlogL{p}{a}(\mfd)$, and where
$\Youngf(s)=\exp(s^a)$, in which we denote the associated Orlicz space
$\expL{a}(\mfd)$. In the case $a=1$, we write $\expLone$ instead of
$\expL1$; in the case $a=p=1$, we write $\LlogLone$ instead of
$\LlogL{1}{1}$.

The Orlicz--Sobolev spaces $\W1{p,a}(\mfd)$ are then defined
analogously to the Sobolev spaces $\W1p(\mfd)$, with the spaces
$\LlogL{p}{a}(\mfd)$ in place of $\Lp p(\mfd)$:
\begin{equation}
  \W1{p,a}(\mfd)=\left\{
  f\in\LlogL{p}{a}(\mfd)
  \,:
  \grad f\in\LlogL{p}{a}(\mfd)
    \right\}
\end{equation}
with the gradient understood in the weak sense.

We recall a number of results concerning dual spaces and continuous
embeddings that are relevant to us:
\begin{prop}[$\expLone\cong\dual{\LlogLone}$]
  \label{prop:expLone-dual-LlogLone}
  For $f\in\expLone(\mfd)$ and $h\in\LlogLone(\mfd)$, we have that
  \begin{equation}
    \int_{\mfd}fh\dv\le C_\mfd\norm[\expLone(\mfd)]{f}\norm[\LlogLone(\mfd)]{h}.
  \end{equation}
\end{prop}
Using this relation, we identify $\expLone(\mfd)$ with the dual of
$\LlogLone(\mfd)$.

\begin{prop}[$\W1{2,-\half}\hookrightarrow\expLone$]
  \label{prop:expL-embed}
  We have the continuous embedding
  \begin{equation}
    \label{eqn:expL-embed}
    \W1{2,-\half}(\mfd)\to\expLone(\mfd),
  \end{equation}
  in which the right-hand side is optimal among Orlicz spaces; this is
  an instance of \cite[Example 1]{cianchi} with $p=q=1$.
\end{prop}
\begin{prop}[Embedding of Orlicz--Sobolev spaces]
  \label{prop:orlicz-sobolev-embedding}
  For every $p\ge1$, $a\ge 0$,and $\eps>0$, we have the continuous
  embeddings
  \begin{equation}
    \label{eqn:orlicz-sobolev-embedding}
    \W1{p+\eps}(\mfd)
    \subset
    \W1{p,a}(\mfd)
    \subset
    \W1p(\mfd)
    \subset
    \W1{p,-a}(\mfd)
    \subset
    \W1{p-\eps}(\mfd)
  \end{equation}
  from which follows also embeddings of the respective dual spaces,
  with the inclusions going in the opposite direction.
\end{prop}

\section{Controlled homogenisation for the Robin problem}
\label{sec:robin-homogenisation}
\subsection{Heuristics}
\label{ssec:outline-of-proof}
Let us begin by considering the setting of a closed Riemannian
manifold $\mfd$, and provide some intuition and motivation for the
construction we shall present in full in the next section.

For $\dom\subseteq\mfd$ and $\bdparam:\bd\dom\to\R$ essentially
bounded, a nonzero $\yu\in\W12(\dom)$ is a Robin eigenfunction for the
eigenvalue $\lam[]{\Rob}(\dom,\bdparam)$ if, for any $v\in\W12(\mfd)$,
\begin{equation}
  \label{eqn:weak-robin-closed}
  \int_{\dom}\grad\yu\cdot\grad v\dv
  +
  \int_{\bd\dom}\yu v\bdparam\dA
  =
  \lam[]{\Rob}
  \int_{\dom}\yu v\dv.
\end{equation}
We define the \emph{Rayleigh quotient} for the Robin problem to be
\begin{equation}
  \label{eqn:robin-rayleigh-closed}
  \Rayleigh{\Rob}(\dom,\bdparam, f)
  =
  \frac{
    \int_{\dom}\abs{\grad f}^2\dv
    +
    \int_{\bd\dom}f^2\bdparam\dA
  }{\int_{\dom} f^2\dv}
\end{equation}
and recall the variational characterisation of Robin eigenvalues
\begin{equation}
  \label{eqn:robin-minimax-closed}
  \lam{\Rob}(\dom,\bdparam)
  =
  \inf_{\substack{F\subseteq\W12(\dom)\\\dim F=k}}\;
  \sup_{f\in F\setminus\{0\}}
  \Rayleigh{\Rob}(\dom,\bdparam, f).
\end{equation}
One similarly formulates the weak Schr\"odinger problem on $\mfd$ for
potential $\Vee$:
\begin{equation}
  \label{eqn:weak-schroedinger-closed}
  \int_{\mfd}\grad \yu\cdot\grad v\dv
  +
  \int_{\mfd}\yu v\Vee\dv
  =
  \lam[]{\Sch}
  \int_{\mfd}\yu v\dv
\end{equation}
for which we have the Rayleigh quotient
\begin{equation}
  \label{eqn:schroedinger-rayleigh-closed}
  \Rayleigh{\Sch}(\mfd,\Vee,f)
  =
  \frac{\int_\mfd\abs{\grad f}^2\dv+\int_\mfd f^2\Vee\dv}{\int_\mfd f^2\dv}
\end{equation}
and whose eigenvalues have the variational characterisation
\begin{equation}
  \label{eqn:schroedinger-minimax-closed}
  \lam{\Sch}(\mfd,\Vee)
  =
  \inf_{\substack{F\subseteq\W12(\mfd)\\\dim F=k}}\;
  \sup_{f\in F\setminus\{0\}}
  \Rayleigh{\Sch}(\mfd,\Vee, f).
\end{equation}
Comparing these pairs of expressions, we are motivated to attempt to
construct domains $\dome[\eps]$ and boundary parameters
$\bdparame[\eps]$ parameterised by $\eps>0$, such that they satisfy
the following properties:
\begin{enumerate}[(i)]
\item
  \emph{As $\eps$ tends to zero, $\dome[\eps]$ eventually exhausts all
  of $\mfd$}. Combined with control of the norm of a family of
  extensions $\gext[\eps]:\W12({\dome[\eps]})\to\W12(\mfd)$,
  this is to the effect that the terms of the form
  $\int_{\dome[\eps]}\dum\dv$ in the expressions
  ~\eqref{eqn:weak-robin-closed} and~\eqref{eqn:robin-rayleigh-closed}
  approach the corresponding terms
  in~\eqref{eqn:weak-schroedinger-closed} and
  ~\eqref{eqn:schroedinger-rayleigh-closed} of the form
  $\int_\mfd\dum\dv$.
\item
  \emph{In the limit, integrating on $\bd\dome[\eps]$
  weighted by the boundary parameter~$\bdparame[\eps]$
  resembles integrating over $\mfd$ weighted by the potential $\Vee$,}
  to the effect that terms of the form
  $\int_{\bd\dome[\eps]}\dum\bdparame[\eps]\dA$
  in~\eqref{eqn:weak-robin-closed} and
  ~\eqref{eqn:robin-rayleigh-closed} approach the corresponding terms
  in ~\eqref{eqn:schroedinger-rayleigh-closed}
  and~\eqref{eqn:weak-schroedinger-closed} of the form
  $\int_\mfd\dum\Vee\dA$.
\end{enumerate}
It is in aiming to satisfy both of these conditions that one is led to
consider a family of domains with many holes, constructed by removing
evenly distributed balls from $\mfd$ whose number grow and radii
shrink as $\eps$ decreases. To satisfy (i), one requires the radii to
decrease sufficiently rapidly as $\eps\tendsto0$. To satisfy (ii), one
adjusts the radius of each hole to capture the weight assigned by
$\Vee$ in a small region near each hole.  This is the mechanism which
permits us the freedom in choosing the constant $\bdconst>0$ as
appears in the statements of Theorems~\ref{thm:main-nonneg}
and~\ref{thm:main}: an increase in the magnitude of the boundary
parameter can be compensated for by decreasing the perimeter of the
holes accordingly to leave the weighted integral unchanged, and vice
versa.

In the case where $\mfd$ is a Riemannian manifold with Lipschitz
boundary, the same strategy can be applied with few modifications. If
the boundary of the holes $\holes[\eps]$ does not intersect the
boundary of $\mfd$, we observe then that the weak formulations
~\eqref{eqn:weak-schroedinger} and~\eqref{eqn:weak-robin} differ
from~\eqref{eqn:weak-schroedinger-closed}
and~\eqref{eqn:weak-robin-closed} only by the term originating from
the boundary conditions on $\bd\mfd$, which is identical between the
Robin and Schr\"odinger problems. Thus adapting the argument sketched
for closed manifolds to work in the case of nonempty boundary amounts
to ensuring that one can avoid placing holes near the boundary while
still having the convergence of $\bdparame[\eps]\bd\holes[\eps]\area$
to $\Vee\volume$ in an appropriate sense.

To fix notation, let us write down the Rayleigh quotients for the
Schr\"odinger problem~\eqref{eqn:weak-schroedinger}, which is
\begin{equation}
  \label{eqn:schroedinger-rayleigh}
  \Rayleigh{\Sch}(\mfd,\Vee,\Bdparam, f)
  =
  \frac{\int_\mfd\abs{\grad f}^2\dv+\int_\mfd f^2\Vee\dv+\int_{\bdmfd}f^2\Bdparam\dA}{\int_\mfd f^2\dv}.
\end{equation}
defined for $f\in\Wz12(\mfd)$, and for the Robin problem~\eqref{eqn:weak-robin}, which is
\begin{equation}
   \label{eqn:robin-rayleigh}
    \Rayleigh{\Rob}({\dome[\eps]},\bdparame[\eps], \Bdparam, f)
    =
    \frac{
      \int_{\dome[\eps]}\abs{\grad f}^2\dv
      +
      \int_{\bd\holes[\eps]}f^2\bdparame[\eps]\dA
      +
      \int_{\bdmfd}f^2\Bdparam\dA
    }{\int_{\dome} f^2\dv}.
\end{equation}
for $f\in\Wz12(\dome[\eps])$.
In this case the infimum in the variational characterisation for the
$k$-th Schr\"odinger eigenvalue $\lam{\Sch}(\mfd,\Vee,\Bdparam)$ is
taken over $k$-dimensional subspaces of $\Wz12(\mfd)$ rather than of
$\W12(\mfd)$. Likewise, the infimum in the variational
characterisation for
$\lam{\Rob}(\dome[\eps],\bdparame[\eps],\Bdparam)$ is taken over
$\Wz12(\dome[\eps])$.

\subsection{The homogenisation construction}
\label{ssec:homogenisation-construction}
We now describe the homogenisation construction in detail. Consider
the Schr\"odinger problem~\eqref{eqn:schroedinger-1} on a Riemannian
manifold $\mfd$ with potential $\Vee\in\Cn1(\mfd)$ and boundary
conditions~\eqref{eqn:schroedinger-2}. For each choice of two positive
parameters $\eps$ and $\bdconst$, we give in this section the
construction of a domain $\dome=\mfd\setminus\holes\subseteq\mfd$. For
those $\eps,\bdconst$ such that $\cl{\holes}$ does not intersect
$\bd\mfd$, we give the construction also of the boundary parameter
$\bdparame:\bd\holes\to\R$.
Later, we shall extract the witnesses to each of our main results by
specifying appropriate choices of the parameter $\bdconst$ as a
function of $\eps$.

The construction begins with a \emph{maximal $\eps$-separated set}
$\sites\subseteq\mfd$, which we recall is a set which is maximal with
respect to inclusion among those subsets $S\subseteq\mfd$ with the
property that $\dist(s_1,s_2)\ge\eps$ for all distinct $s_1,s_2\in S$.
We consider the Voronoi tessellation associated with $\sites$, in
which the Voronoi cell $\cell$ at $\site\in\sites$ is defined
\begin{equation}
  \cell=
  \{x\in\mfd
  \,:
  \dist(x,\site)\le\dist(x,\site')
  \text{ for all }\site'\in\sites
  \}.
\end{equation}
We now write $\sites$ as the union of three disjoint subsets:
\begin{itemize}
\item Let $\bsites$ be all the sites whose associated Voronoi cells
  intersect the boundary of $\mfd$:
  \begin{equation}
    \label{eqn:bsites-dfn}
    \bsites=\{
    \site\in\sites
    \,:
    \cell\cap\bd\mfd\ne\varnothing
    \}.
  \end{equation}
\item Let $\psites$ be the set of all remaining sites where
  $\abs{\Vee}$ on the associated Voronoi cell is pointwise bounded
  away from zero by at least $\eps^\half$
  \begin{equation}
    \label{eqn:psites-dfn}
    \psites=\{
    \site\in\sites\setminus\bsites
    \,:
    \abs{\Vee(x)}>\eps^\half
    \text{ for all }
    x\in\cell
    \}.
  \end{equation}
\item Let $\osites$ be the set of all remaining sites: $\osites=\sites\setminus(\psites\cup\bsites)$.
\end{itemize}
\begin{rmk}
  The specific exponent $\half$ appearing in~\eqref{eqn:psites-dfn} is
  inconsequential, and is here chosen explicitly only as a matter of
  definiteness. From the proof of
  Proposition~\ref{prop:cell-integral-estimate}, we see that any
  positive exponent strictly less than~$1$ suffices for the same
  purposes.
\end{rmk}
The domain $\dome$ is constructed by removing a metric ball
$\hole$ from each $\site\in\psites$, where $\radius\ge0$ is defined by
requiring
\begin{equation}
  \label{eqn:radius-dfn}
  \Area{\bd\hole}=\frac{1}{\bdconst}\abs{\Vee(\site)}\Vol(\cell).
\end{equation}
We define the \emph{perforated cell} $\pcell$ at $\site$ by
$\pcell=\cell\setminus\hole$. For $\site\in\sites\setminus\psites$, we
do not perforate the corresponding cell; instead we set
$\pcell=\cell$, and adopt the convention of taking $\radius=0$ and
$\hole$ to be empty. We write $\holes$ for the union of all the holes
\begin{equation}
  \label{eqn:holes-dfn}
  \holes=\bigcup_{\site\in\psites}\hole
\end{equation}
and finally define the \emph{perforated domain} $\dome$ as the
complement of $\holes$ in $\mfd$.
\begin{figure}[h]
  \includegraphics{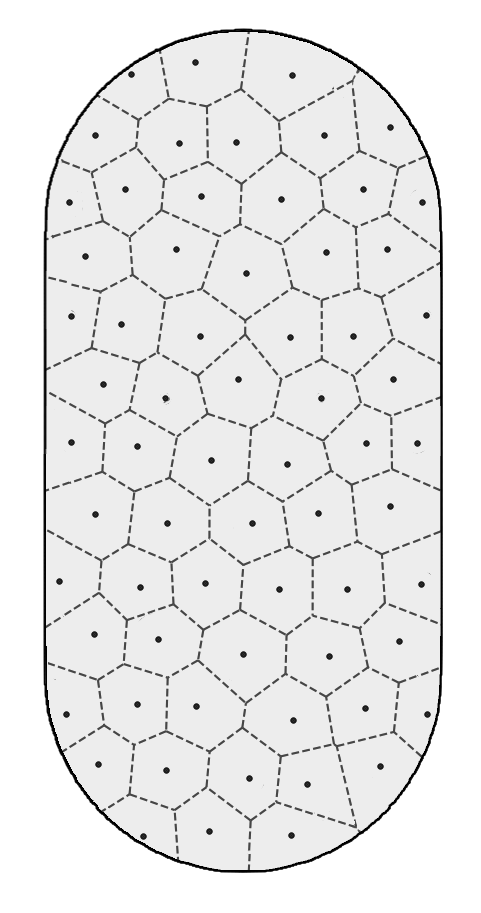}
  \includegraphics{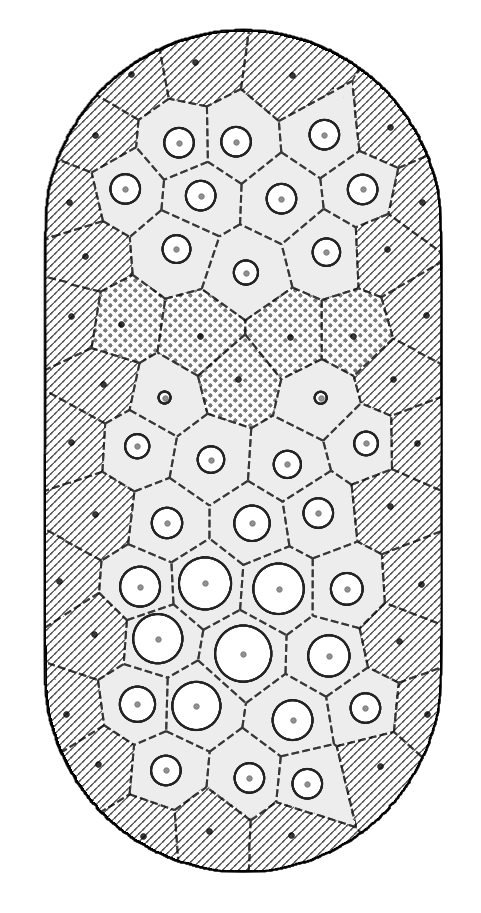}
  \includegraphics{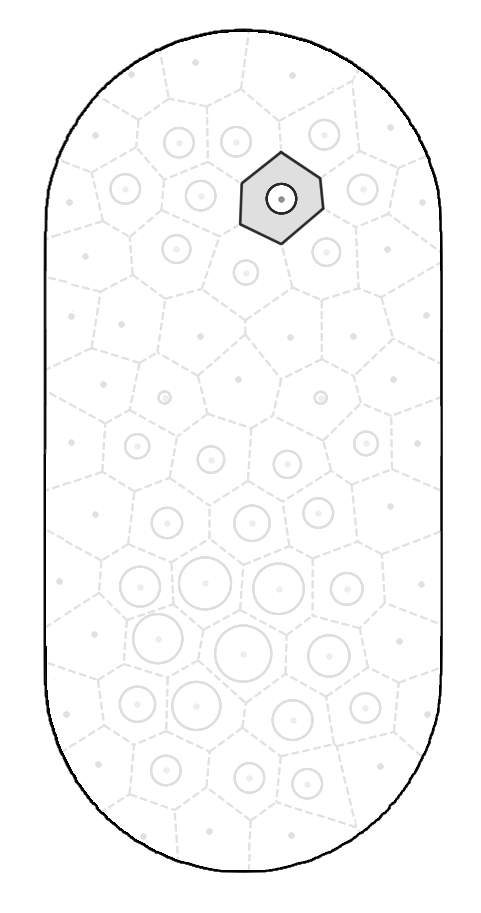}
  \newcommand{\textureblock}[1]{\raisebox{-0.2ex}{\includegraphics[width=1em]{picture/texture-#1}}}
  \label{fig:homogenisaion-construction}
  \caption{Example of the construction of a perforated domain in
    dimension $2$, for a potential $\Vee$ which changes sign across a
    region lying in the middle of the domain. \textbf{Left:} the
    domain; within it a maximal $\eps$-separated set $\sites$ and the
    associated Voronoi tessellation, with the boundaries between cells
    indicated by dotted lines. \textbf{Centre:} the domain with
    perforations, partitioned into three sets:
    \textureblock{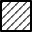}---the cells associated with $\bsites$ ,
    which intersect the boundary. \textureblock{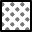}---the cells
    associated with $\osites$, where $\Vee$ can be close to
    zero. \textureblock{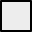}---the cells associated with $\psites$;
    these lie entirely in the interior of $\dom$, where $\Vee$ is
    bounded sufficiently far away from zero. From all such cells we
    remove a metric ball whose radius is given
    by~\eqref{eqn:radius-dfn}; the holes are proportionally larger on
    cells where the magnitude of $\Vee$ is larger.  \textbf{Right:}
    a~single perforated cell $\pcell$ highlighted.}
\end{figure}

Let us additionally adopt the following notation: for ${\bullet}\in\{\bd,\ocircle,\odot\}$,
\begin{equation}
  \label{eqn:idom-dfn}
  \dome_{\bullet}=\bigcup_{\site\in S_{\bullet}^\eps}\pcell
\end{equation}
so that in summary
\begin{equation}
  \label{eqn:dom-dfn}
  \dome
  =
  \mfd\setminus\holes
  =
  \bigcup_{\site\in\sites}\pcell
  =
  \bdom\cup\odom\cup\pdom.
\end{equation}
From~\eqref{eqn:radius-dfn}, it follows that
\begin{equation}
  \label{eqn:radius-asymp}
  \radius
  \asymp_{\mfd,\Vee}
  (\bdconst\inv\eps)^{\frac{1}{\mfdim-1}}\eps.
\end{equation}
From \eqref{eqn:radius-asymp} we deduce that so long as
\begin{equation}
  \label{eqn:bdconst-hyp}
  \bdconst\inv=\littleo[\mfd,\Vee](\eps\inv)
  \quad
  \text{as $\eps\tendsto0$}
\end{equation}
we have that $\radius$ shrinks much faster than $\eps$ as
$\eps\tendsto0$. Since, by construction, each cell $\cell$ contains a
ball of radius $\half\eps$ centred at $\site$, this
implies for all sufficiently small $\eps>0$
each hole $\hole$ is indeed contained compactly within
$\cell$, and the union in~\eqref{eqn:holes-dfn} is
disjoint. Eventually we have that $\bd\dome$ is the disjoint union
\begin{equation}
  \label{eqn:bd-union}
  \bd\dome=\bd\mfd\cup\bd\holes=\bd\mfd\cup\left(
  \bigcup_{\site\in\psites}\bd\hole
  \right).
\end{equation}
In this case the perforated domain can be written as the union of all
perforated cells:
\begin{equation}
  \dome=\bigcup_{\site\in\sites}\pcell. 
\end{equation}
And we define $\bdparame$ by setting
\begin{equation}
  \label{eqn:bdparame-dfn}
  \bdparame\at{\bd\hole}=\bdconst\sign(\Vee(\site))
\end{equation}
for each $\site\in\psites$.

Now that we have specified the construction, let us also fix
henceforth the notation $\yuek\in\Wz12(\dome)$ for the $k$-th Robin
eigenfunction, chosen such that $\sequence{\yuek}{k\in\N}$ is an
orthonormal family in~$\Lp2(\dome)$.

We are now ready to state the following proposition, from which will
follow each of the results presented in
Section~\ref{sec:introduction}:
\begin{prop}
  \label{prop:main}
  Let $\mfd$ be a Riemannian manifold of dimension $\mfdim\ge2$, and
  let $\Vee$ be an admissible potential on $\mfd$. Let
  $1<p<\frac{d}{d-1}$, and let $\bdconst=\bdconst(\eps)$ be such that
  \begin{equation}
    \label{eqn:bdconst-range}
    \lim_{\eps\tendsto0}\bdconst\inv\eps=0
    \quad\text{and}\quad
      \lim_{\eps\tendsto0}\bdconst\eps^{\frac{\mfdim(p-1)}{\mfdim-p}}=0
  \end{equation}
  as $\eps\tendsto0$. Let $\dome$,
  $\bdparame$ be as constructed in
  Section~\ref{ssec:homogenisation-construction}.
  Then, as $\eps\tendsto0$,
  \begin{itemize}
  \item
    $\lam{\Rob}(\dome,\bdparame,\Bdparam)\tendsto\lam{\Sch}(\mfd,\Vee,\Bdparam)$,
    and
  \item
    There exist extensions $\Yuek\in\Wz12(\mfd)$ of
    $\yuek\in\Wz12(\dome)$ such that $\Yuek$ converge, in the weak
    topology on $\Wz12(\mfd)$, to the $k$-th Schr\"odinger
    eigenfunction.
  \end{itemize}
\end{prop}
\begin{rmk}
  The first hypothesis in~\eqref{eqn:bdconst-range} is
  precisely~\eqref{eqn:bdconst-hyp}, which constrains $\bdconst$ not to
  decrease too rapidly as $\eps\tendsto0$. We have seen that this
  ensures via~\eqref{eqn:radius-asymp} that the radii $\radius$ shrink
  sufficiently fast for the definition~\ref{eqn:bdparame-dfn} to make
  sense for all sufficiently small $\eps>0$. The second hypothesis
  in~\eqref{eqn:bdconst-range} is that
\begin{equation}
    \bdconst
    =\littleo[\mfd,\Vee,p]\left(
    \eps^{-\frac{\mfdim(p-1)}{\mfdim-p}}
    \right)
\end{equation}
which on the other hand constrains the rate at which $\bdconst$ may
grow as $\eps$ tends to $0$. The reason for the exponent
$-\tfrac{\mfdim(p-1)}{\mfdim-p}$ is made clear in the proof of
Proposition~\ref{prop:measures-converge}.
\end{rmk}

The first two results presented in Section~\ref{sec:introduction} are
readily recovered:
\begin{proof}[Proof of Theorems~\ref{thm:main-nonneg} and~\ref{thm:main} from Proposition~\ref{prop:main}]
Theorem~\ref{thm:main} is obtained from Proposition~\ref{prop:main} by
choosing $\bdconst(\eps)$ to be a constant independent of $\eps$. When
$\mfd$ is closed, $\bd\dome=\bd\holes$; when $\Vee$ is non-negative,
from $\bdparame$ is constant over all of $\bd\holes$
by~\eqref{eqn:bdparame-dfn}. Thus, for closed $\mfd$ and non-negative
$\Vee$, we recover Theorem~\ref{thm:main-nonneg}.
\end{proof}
In Section~\ref{sec:flexibility}, we shall see that the measures
$\mue[\eps]$ in Theorem~\ref{thm:flexibility} correspond to weighted
boundary measures on domains $\dome[\eps,\bdconst(\eps)]$ where we
control the asymptotic behaviour of $\bdconst$ depending on $p_0$ to
achieve the desired limits.

\subsection{Analytic properties of perforated manifolds}
\label{ssec:analytic-properties}
Throughout this section, let $\mfd$ be a Riemannian manifold of
dimension $\mfdim\ge2$, let $\Vee\in\Cn1(\mfd)$, and let
$\dome=\mfd\setminus\holes\subseteq\mfd$ and
$\bdparame:\bd\holes\to\{-\bdconst,\bdconst\}$ be as constructed in
Section \ref{ssec:homogenisation-construction}. Additionally, let us
write
\begin{equation}
  \gee_\scale=\scale^2\gee
\end{equation}
for the metric on $\mfd$ scaled by a constant factor $\scale>0$.

For $\site\in\mfd$ and $0<r<R$, let us write
$\ann{\site}{r}{R}=\ball{\site}{R}\setminus\cl{\ball{\site}{r}}$ for
the open annulus centred at $\site\in\mfd$ with inner and outer radii
$r$ and $R$, respectively. Then we have the following:
\begin{prop}[Trace bound with scaling on annuli]
  \label{prop:trace-with-scaling-annuli}
  Let $1<p<\mfdim$. For $\Rrate>\rrate>0$, let
  $\atrace:\W1p(\ann{\site}{\rrate}{\Rrate},\volume[\gee_\scale])\to\Lp
  p(\bd{\ball{\site}{\rrate}},\area[\gee_\scale])$ be the trace onto
  the inner boundary of the annulus with radii $(\rrate,\Rrate)$.

  Then there exist constants $C_{\mfd,p}>0$ and $\eps_0>0$ such that
  for all $\Rrate<\eps_0$, we have
  \begin{equation}
    \label{eqn:trace-with-scaling-annuli}
    \norm{\atrace}
    \le
    C_{\mfd,p}
    \max\left\{
    \scale^{-\frac{1}{p}}\rrate^{\frac{\mfdim-1}{p}}\Rrate^{-\frac{\mfdim}{p}},
    \scale^{\frac{p-1}{p}}\rrate^{\frac{p-1}{p}}\right\}.
  \end{equation}
\end{prop}
Before we present the proof, let us first mention the reason we
are interested in this result. In the context of the homogenisation
construction we deduce the following proposition, which is thought of
as bounding the trace of $\W1p$ functions on the perforated cell
$\pcell$ on the boundary $\bd{\hole}$ of the hole:
\begin{prop}[Trace bound with scaling on perforated cell]
  \label{prop:trace-with-scaling}
  For $1<p<\mfdim$, and let
  $\trace:\W1p(\pcell,\volume[\gee_\scale])\to\Lp
  p(\bd\hole,\area[\gee_\scale])$ be the trace onto the boundary of
  the hole at~$\site\in\psites$. Then
  \begin{equation}
    \label{eqn:trace-with-scaling}
    \norm{\trace}
    \lesssim_{\mfd,\Vee,p}
    \max\left\{
    \scale^{-\frac{1}{p}}
    \bdconst^{-\frac{1}{p}}
    ,
    \scale^{\frac{p-1}{p}}
    \bdconst^{-\frac{p-1}{p(d-1)}}
    \eps^{\frac{\mfdim(p-1)}{(\mfdim-1)p}}
    \right\}
  \end{equation}
  as $\eps\tendsto0$.
\end{prop}
\begin{proof}[Proof of Proposition~\ref{prop:trace-with-scaling} from Proposition~\ref{prop:trace-with-scaling-annuli}]
  Recall~\eqref{eqn:radius-asymp}, which is that
  $\radius\asymp_{\mfd,\Vee}\bdconst^{-\frac{1}{\mfdim-1}}\eps^{\frac{\mfdim}{\mfdim-1}}$. Let
  us also define $\Radius=\half\eps$. For all sufficiently small
  $\eps>0$, we then have the inclusions
  $\ball{\site}{\radius}\subset\ball{\site}{\Radius}\subset\cell$.
  Applying Proposition~\ref{prop:trace-with-scaling-annuli} with
  $\rrate=\radius$ and $\Rrate=\Radius$, we find that
  \begin{equation}
    \norm[\Lp p(\bd\hole,{\area[\gee_\scale]})]{f}
    \lesssim_{\mfd,\Vee,p}
    \max\left\{
    \scale^{-\frac{1}{p}}
    \bdconst^{-\frac{1}{p}}  
,
    \scale^{\frac{p-1}{p}}
    \bdconst^{-\frac{p-1}{p(d-1)}}
    \eps^{\frac{\mfdim(p-1)}{(\mfdim-1)p}} \right\}
    \norm[\W1p({\ann{\site}{\radius}{\Radius}},{\volume[\gee_\scale]})]{f}.
  \end{equation}
  Finally, we recover the desired inequality by replacing
  $\norm[\W1p({\ann{\site}{\radius}{\Radius}},{\volume[\gee_\scale]})]{f}$
  on the right-hand side with the norm on a larger set $\norm[\W1p
    (\pcell,{\volume[\gee_\scale]})]{f}$.
\end{proof}
\begin{rmk}
  The additional dependence on $\Vee$ in the asymptotic relation in
  ~\eqref{eqn:trace-with-scaling} compared to
  ~\eqref{eqn:trace-with-scaling-annuli} enters
  via~\eqref{eqn:radius-dfn}, which is that the radii
    of the holes depend on the potential $\Vee$.
\end{rmk}
Towards the proof of Proposition~\ref{prop:trace-with-scaling-annuli},
we first present an analogous lemma in the Euclidean setting. Let us
write $\Eball{r}\subseteq\R^\mfdim$ for the ball at the origin with radius
$r>0$, and $\Eann{r}{R}$ for the annulus
$\Eball{R}\setminus\cl{\Eball{r}}\subseteq\R^\mfdim$. Then
\begin{lem}[Trace bound on annuli in $\R^\mfdim$]
  \label{lem:euc-trace}
  Let $0<r\le R\le 1$, and $1<p<\mfdim$. Then there exists $C_{p,\mfdim}>0$
  such that for all $h\in\W1p(\Eball{R})$:
  \begin{equation}
    \label{eqn:euc-trace}
    \norm[\Lp p(\bd{\Eball{r}})]{h}^p
    \le
    C_{p,\mfdim}
    \left(
    r^{\mfdim-1}R^{-\mfdim}
    \norm[\Lp p(\Eann{r}{R})]{h}^p
    +
    r^{p-1}
    \norm[\Lp p(\Eann{r}{R})]{\grad[] h}^p
    \right).
  \end{equation}
\end{lem}
The proof is similar to that of \cite[Lemma~6.4]{gkl}. We compare
Lemma~\ref{lem:euc-trace} and Proposition
\ref{prop:trace-with-scaling} also to
\cite[Lemma~3.3]{girouard-lagace}, which is the analogous statement
for the case~$p=2$.
\begin{proof}[Proof of Lemma~\ref{lem:euc-trace}]
  By density it suffices to prove inequality~\eqref{eqn:euc-trace} for
  smooth $h:\Eball{R}\to\R$. Working in spherical co-ordinates, we
  define the radially constant function
  $\tilde{h}(\rho,\theta)=h(r,\theta)$ for all $\rho>0$ and
  $\theta\in\sphere[\mfdim-1]$, and define $H=h-\tilde{h}$. A~direct
  computation gives us that
  \begin{equation}
    \begin{aligned}
      \norm[\Lp p(\Eann{r}{R})]{\tilde{h}}^p
      &=
      \int_{\sphere[\mfdim-1]}\int_r^R
      \abs{\tilde{h}(\rho,\theta)}^p
      \rho^{\mfdim-1}
      \,\dee{\rho}
      \,\dee{\theta}\\
      &=
      \int_{\sphere[\mfdim-1]}
      h(r,\theta)^p
      \underbrace{
        \frac{1}{\mfdim}\left(R^\mfdim-r^\mfdim\right)
      }_{\ge R^\mfdim/\mfdim}
      \,\dee{\theta}\\
      &\ge
      \frac{R^\mfdim r^{1-\mfdim}}{\mfdim}
      \underbrace{
        \int_{\sphere[\mfdim-1]}
        h(r,\theta)^p
        r^{\mfdim-1}
        \,\dee{\theta}
      }_{=\norm[\Lp p(\bd{\Eball{r}})]{h}^p}.
    \end{aligned}
  \end{equation}
  From this, it follows that
  \begin{equation}
    \label{eqn:euc-trace-1}
    \begin{aligned}
      \norm[\Lp p(\bd{\Eball{r}})]{h}^p
      &\le
      \mfdim R^{-\mfdim}r^{\mfdim-1}\norm[\Lp p(\Eann{r}{R})]{\tilde{h}}^p\\
      &=
      \mfdim R^{-\mfdim}r^{\mfdim-1}\norm[\Lp p(\Eann{r}{R})]{h-H}^p\\     
      &\le
      C_{p,\mfdim}
      R^{-\mfdim}r^{\mfdim-1}\left(
      \norm[\Lp p(\Eann{r}{R})]{h}^p
      +
      \norm[\Lp p(\Eann{r}{R})]{H}^p
      \right).
    \end{aligned}
  \end{equation}
  By construction $H$ vanishes on $\bd{\Eball{r}}$ and
  $\pd{r}H=\pd{r}h$, so using the fundamental theorem of calculus and
  then applying H\"older's inequality with exponents $p$ and
  $\frac{p}{p-1}$, we have
  \begin{equation}
    \label{eqn:euc-trace-2}
    \begin{aligned}
      \abs{H(\rho,\theta)}^p
      &\le
      \bigg(
      \int_r^\rho\abs{\pd{s}h(s,\theta)}
      \,\dee{s}
      \bigg)^p\\
      &\le
      \left(
      \norm[{\Lp p([r,\rho])}]{s^{\frac{\mfdim-1}{p}}\pd{s}h}^{\frac{1}{p}}
      \norm[{\Lp{\frac{p}{p-1}}([r,\rho])}]{s^{\frac{1-\mfdim}{p}}}^{\frac{p-1}{p}}
      \right)^p\\
      &=
      \left(
      \int_r^\rho\abs{\pd{s}h(s,\theta)}^ps^{\mfdim-1}
      \,\dee{s}
      \right)
      \left(
      \int_r^\rho s^{\frac{1-\mfdim}{p-1}}
      \,\dee{s}
      \right)^{p-1}\\
      &=
      \left(
      \int_r^\rho\abs{\pd{s}h(s,\theta)}^ps^{\mfdim-1}
      \,\dee{s}
      \right)
      \Bigg(
      \frac{p-1}{\mfdim-p}
      \Big(
      \underbrace{r^{\frac{p-\mfdim}{p-1}}-\rho^{\frac{p-\mfdim}{p-1}}}_{\le r^{\frac{p-\mfdim}{p-1}}}
      \Big)
      \Bigg)^{p-1}\\
      &\le
      C_{p,\mfdim}
      \left(
      \int_r^\rho\abs{\pd{s}h(s,\theta)}^ps^{\mfdim-1}
      \,\dee{s}
      \right)
      r^{p-\mfdim}
    \end{aligned}
  \end{equation}
  The estimate on the final line depends on the hypothesis that
  $1<p<d$, which determines the sign of the
  exponent~$\frac{p-\mfdim}{p-1}$.

  Then, integrating both sides of this inequality over $\Eann{r}{R}$,
  we recover a term that is bounded by the $\Lp p$-norm of $\grad[] h$
  in~$\Eann{r}{R}$:
  \begin{equation}
    \label{eqn:euc-trace-3}
    \begin{aligned}
      \norm[\Lp p(\Eann{r}{R})]{H}^p
      &=
      \int_r^R
      \int_{\sphere[\mfdim-1]}
      \abs{H(\rho,\theta)}^p
      \rho^{\mfdim-1}
      \,\dee{\theta}
      \,\dee{\rho}\\
      &\le
      \int_r^R
      \int_{\sphere[\mfdim-1]}
      C_{p,\mfdim}
      \left(
      \int_r^\rho\abs{\pd{s}h(s,\theta)}^ps^{\mfdim-1}
      \,\dee{s}
      \right)
      r^{p-\mfdim}
      \rho^{\mfdim-1}
      \,\dee{\theta}
      \,\dee{\rho}\\      
      &\le
      C_{p,\mfdim}
      \norm[\Lp p(\Eann{r}{R})]{\pd{s}h}^p
      \int_r^R
      r^{p-\mfdim}
      \rho^{\mfdim-1}
      \,\dee{\rho}\\
      &\le
      C_{p,\mfdim}
      r^{p-\mfdim}R^\mfdim
      \norm[\Lp p(\Eann{r}{R})]{\grad[] h}^p      
    \end{aligned}.
  \end{equation}
  Inserting~\eqref{eqn:euc-trace-2} and~\eqref{eqn:euc-trace-3} into
 ~\eqref{eqn:euc-trace-1} yields the desired inequality
 ~\eqref{eqn:euc-trace}.
\end{proof}
The proof of Proposition~\ref{prop:trace-with-scaling-annuli} now
amounts to passing into normal co-ordinates and applying
Lemma~\ref{lem:euc-trace}, taking into consideration the scaling of
the metric.
\begin{proof}[Proof of Proposition~\ref{prop:trace-with-scaling-annuli}]
  Again by density, it suffices by to achieve the inequality for smooth
  functions. By hypothesis, $\lim_{\eps\tendsto0}\Rrate=0$, so there is
  $\eps_0>0$ such that, for all $\eps<\eps_0$, $\Rrate(\eps)$ is no
  greater than the injectivity radius of $\mfd$. In particular, there is
  a chart for geodesic polar co-ordinates
  $\chart:[0,\Rrate(\eps_0))\times\sphere[\mfdim-1]\to\ball{\Rrate(\eps_0)}{\site}\subseteq
    M$. Let $f$ be a smooth function on $\ann{\site}{\rrate}{\Rrate}$.
    We compute:
    \begin{equation}
      \label{eqn:trace-scale-lhs}
      \begin{aligned}
        \norm[\Lp p(\bd{\ball{\rrate}{\site}},{\area[\gee_\scale]})]{f}^p
        &=
        \scale^{\mfdim-1}(1+\bigO[\mfd](\eps))
        \int_{\sphere[\mfdim-1]}\abs{f\circ\chart(\rrate,\theta)}^p
        \,\dee{\theta}\\
        &\asymp_{\mfd}
        \scale^{\mfdim-1}
        \norm[\Lp p(\bd{\Eball{\rrate}})]{\pullb{\chart}f}^p.\\
      \end{aligned}
    \end{equation}
    Similar computations give us also that
    \begin{equation}
      \label{eqn:trace-scale-rhs}
      \begin{aligned}
        \norm[\Lp p({\ball{\Rrate}{\site}},{\volume[\gee_\scale]})]{f}^p
        &\asymp_{\mfd}
        \scale^{\mfdim}\norm[\Lp p(\Eball{\Rrate})]{\pullb{\chart}f}^p
        \\
        \norm[\Lp p({\ball{\Rrate}{\site}},{\volume[\gee_\scale]})]{\grad f}^p
        &\asymp_{\mfd}
        \scale^{\mfdim-p}\norm[\Lp p(\Eball{\Rrate})]{\grad[] (\pullb{\chart}f)}^p.
      \end{aligned}
    \end{equation}
    Chaining together~\eqref{eqn:trace-scale-lhs},~\eqref{eqn:euc-trace},
    and~\eqref{eqn:trace-scale-rhs}, then recalling the definition of the
    $\W1p$-norm, we have that
    \begin{equation}
      \label{eqn:trace-scale-ineq}
      \begin{aligned}
        \norm[\Lp p(\bd\hole,{\area[\gee_\scale]})]{f}^p
        &\asymp_{\mfd}
        \scale^{\mfdim-1}
        \norm[\Lp p(\bd{\Eball{\rrate}})]{\pullb{\chart}f}^p\\
        &\asymp_{\mfd}
        C_{p,d}
        \scale^{\mfdim-1}
        \left(
        \rrate^{\mfdim-1}\Rrate^{-\mfdim}\norm[\Lp p(\Eann{\rrate}{\Rrate})]{\pullb{\chart}f}^p
        +
        \rrate^{p-1}
        \norm[\Lp p(\Eann{\rrate}{\Rrate})]{\grad[](\pullb{\chart}f)}^p
        \right)\\
        &\asymp_{\mfd,p}
        \scale^{\mfdim-1}
        \left(
        \rrate^{\mfdim-1}\Rrate^{-\mfdim}
        \scale^{-\mfdim}
        \norm[\Lp p(\ann{\site}{\rrate}{\Rrate})]{f}^p
        +
        \rrate^{p-1}
        \scale^{p-\mfdim}
        \norm[\Lp p(\ann{\site}{\rrate}{\Rrate})]{\grad f}^p
        \right)\\
        &\asymp_{\mfd,p}
        \scale\inv\rrate^{\mfdim-1}\Rrate^{-\mfdim}
        \norm[\Lp p(\ann{\site}{\rrate}{\Rrate})]{f}^p
        +
        \scale^{p-1}\rrate^{p-1}
        \norm[\Lp p(\ann{\site}{\rrate}{\Rrate})]{\grad f}^p\\    
        &\lesssim_{\mfd,p}
        \max\left\{\scale\inv\rrate^{\mfdim-1}\Rrate^{-\mfdim},\scale^{p-1}\rrate^{p-1}\right\}
        \norm[\W1p(\ann{\site}{\rrate}{\Rrate},{\volume[\gee_\scale]})]{f}^p
      \end{aligned}
    \end{equation}
    which gives us the desired inequality
   ~\eqref{eqn:trace-with-scaling-annuli}.
\end{proof}
\begin{rmk}
Proposition~\ref{prop:trace-with-scaling} is stated in the setting of
a single Voronoi cell. However, note that the estimate
\eqref{eqn:trace-with-scaling} is uniform over $\site\in\psites$; this
is necessary for its use in the proof of
Proposition~\ref{prop:measures-converge}.
\end{rmk}
The next two results are, or follow readily from, results that are
proven elsewhere. We recall them to bring the statements into the form
in which they are used in the present work.

The first of these is the existence of family of maps that extend
functions on $\dome$ over the holes $\holes$ to produce functions on
$\mfd$ in such a way that remains bounded in the limit. This follows
quickly, for example, from \cite[Lemma~3.4]{girouard-lagace}, where it
is shown that one can extend a function on a geodesic annulus in this
way over the ball bounded by its inner radius, by means of the
harmonic extension. From this we deduce
\begin{lem}[Uniformly bounded extension over holes]
  \label{lem:extension-operator}
  There exist a family of continuous linear maps
  $\gext:\Wz1p(\dome)\to\Wz1p(\mfd)$ such that
  \begin{equation}
    (\gext f)\at{\dome}=f,
  \end{equation}
  and such that there exists $\eps_0>0$,
  $C_{\mfd}>0$ so that for all $0<\eps<\eps_0$,
  $\norm{\gext}\le C_{\mfd}$.
\end{lem}
See also \cite[Lemma~4.3]{anne-post} for a proof of an analogous
result in the Euclidean setting, where the computations are presented
in detail.

The following proposition concerns the $p$-Poincar\'e inequality on
Voronoi cells. Recall that the $p$-Poincar\'e inequality is satisfied
for a domain $\Omega\subseteq\mfd$ on a Riemannian manifold $(\mfd,g)$
if there exists a constant $\Kay(\Omega,g)$ such that for all
$f\in\W1p(\Omega)$
\begin{equation}
  \label{eqn:p-poincare}
  \int_\Omega\abs{f-m_f}^p\dv\le\Kay(\Omega,\gee)\int_\Omega\abs{\grad f}^p\dv.
\end{equation}
where $m_f=\Vol(\Omega)\inv\int_\Omega f\dv$.  The infimum of all such
$\Kay(\Omega,\gee)$ is the \emph{$p$-Poincar\'e constant}; without
introducing any ambiguity, we choose to also call this infimum
$\Kay(\Omega,\gee)$.
\begin{prop}[$p$-Poincar\'e constant on Voronoi cells with scaled metric]
  \label{prop:p-poincare-constant}
  For $1\le p\le\infty$, there exist constants $\eps_0>0$ and
  $C_{\mfd,p}>0$ such that for all $0<\eps<\eps_0$,
  \begin{equation}
    \Kay(\cell,\volume[\gee_\scale])\le C_{\mfd,p}\scale^p\eps^p.
  \end{equation}
\end{prop}
To prove Proposition~\ref{prop:p-poincare-constant}, we
recall~\cite[Proposition 2.2]{hkp}, whose statement we reproduce
below:
\begin{prop}
  \label{prop:hkp}
  Let $(\mfd,g)$ be a complete $\mfdim$-dimensional Riemannian manifold whose
  Ricci curvature is bounded below by $\mathrm{Ric}_g\ge-(\mfdim-1)\kappa$
  for some $\kappa\ge 0$. Then for any $p\ge 1$, geodesically convex
  domain $\Omega\subseteq \mfd$, and ${f\in\Cn{\infty}(\Omega)}$, the
  following inequality holds:
  \begin{equation}
    \int_{\ball{s}{R}\cap\Omega}\abs{f-m_f}^p\dv
    \le
    C_{n,p} R^p e^{(n-1)R\sqrt{\kappa}}
    \int_{\ball{s}{2R}\cap\Omega}\abs{\grad f}^p\dv
  \end{equation}
  for any $R>0$, $s\in \mfd$; in which
  \begin{equation}
    m_f=\Vol(\ball{s}{R}\cap\Omega)\inv\int_{\ball{s}{R}\cap\Omega} f\dv
  \end{equation}
  is the mean of $f$ on $\ball{s}{R}\cap\Omega$.
\end{prop}
\begin{proof}[Proof of Proposition~\ref{prop:p-poincare-constant}]
  From the characterisation~\eqref{eqn:p-poincare} of the $p$-Poincar\'e
  constant it is evident that
  $\Kay(\cell,\volume)=\scale^p\Kay(\cell,\volume[\gee_\scale])$, so to
  prove Proposition~\ref{prop:p-poincare-constant} it suffices to prove
  that $\Kay(\cell,\volume)\lesssim_{\mfd,p}\eps^p$ as~$\eps\tendsto0$.
  
  Since $\mfd$ is compact, the hypothesis on bounded Ricci curvature
  in Proposition~\ref{prop:hkp} is satisfied for some $\kappa\ge
  0$. By construction, the Voronoi cell $\cell$ is contained within
  the geodesic ball of radius $R=3\eps$ centred at
  $\site\in\mfd$. Finally, for all small enough $\eps$ we have that
  $\cell$ is geodesically convex.  The desired result now follows from
  invoking Proposition ~\ref{prop:hkp} with the choices $\Omega=\cell$
  and $R=3\eps$.
\end{proof}
Let us now define the measure $\mues$ on $\mfd$ to be the push-forward
of the boundary measure $\area[\gee_\scale]$ on $\bd\holes$ with
respect to the metric $\gee_\scale$, weighted by $\bdparame$; that is,
%
\begin{equation}
  \label{eqn:mues-dfn}
  \mues=\bdparame\area[\gee_\scale].
\end{equation}
When $\scale=1$, we write simply $\mue$ in place of~$\mues$.

The final results in this section concern the properties of the
measures $\mue$. Recall from the introduction that we expect $\mue$ to
converge to $\Vee\volume$ in the limit $\eps\tendsto0$; this is made
precise in the several propositions that follow.
\begin{prop}[Estimate for integrals on cells]
  \label{prop:cell-integral-estimate}
  For all $\delta>0$, there exists $\eps_0>0$ such that for all
  $0<\eps<\eps_0$ and for all $\site\in\psites$ we have that
  \begin{equation}
    \label{eqn:cell-integral-estimate}
    \bigabs{1-
      \frac{\int_{\bd\hole}\bdparame\dA}{\int_{\cell}\Vee\dv}}<\delta.
  \end{equation}
\end{prop}
\begin{proof}
  For $\site\in\psites$, we rewrite
  \begin{equation}
    \label{eqn:Vf-continuity}
    1-\frac{\int_{\bd\hole}\bdparame\dA}{\int_{\cell}\Vee\dv}
    =
    \frac{\int_{\cell}\Vee-\Vee(s)\dv}{\int_{\cell}\Vee\dv}+
    \frac{\int_{\cell}\Vee(s)\dv-\int_{\bd\hole}\bdparame\dA}{
      \int_{\cell}\Vee\dv}
  \end{equation}
  and see that the second term on the right-hand side vanishes, by
 ~\eqref{eqn:radius-dfn}. By construction of $\psites$ we have that
  $\abs{\Vee(x)}>\eps^\half$ for all $x\in\cell$, so
  \begin{equation}
    \bigabs{\int_{\cell}\Vee\dv}\ge\eps^\half\Vol(\cell)\asymp_{\mfd,\Vee}\eps^{\mfdim+\half}.
  \end{equation}
  On the other hand,
  \begin{equation}
    \bigabs{\int_{\cell}\Vee-\Vee(s)\dv}\lesssim_{\mfd,\Vee}\eps^{\mfdim+1}.
  \end{equation}
  So the ratio of these two quantities vanishes in the limit
  $\eps\tendsto 0$. Substituting this into~\eqref{eqn:Vf-continuity},
  we conclude that
  \begin{equation}
    \frac{\int_{\bd\hole}\bdparame\dA}{\int_{\cell}\Vee\dv}\tendsto1
    \quad\text{as}\quad
    \eps\tendsto0
  \end{equation}
  which is the desired result.
\end{proof}
Note that \eqref{eqn:cell-integral-estimate} is uniform over
$\site\in\psites$. In particular, by aggregating this estimate across
all $\site\in\psites$, we deduce the following:
\begin{cor}[Weak-$*$ convergence of measures]
  \label{cor:measures-weak-converge}
  The measures $\mue$ converge in the weak-$*$ sense to~$\Vee\volume$.
\end{cor}
For our purposes, it is necessary to establish this convergence of
measures in a stronger sense. Rather than present a direct proof of
Corollary~\ref{cor:measures-weak-converge}, we obtain it as a
consequence of the following proposition:
\begin{prop}[Convergence of measures in the dual of $\W1p(\mfd)$]
  \label{prop:measures-converge}
  Let $1<p<\infty$. Then, for $\bdconst:\Rpos\to\Rpos$
  satisfying~\eqref{eqn:bdconst-range},
  we have that
  \begin{equation}
    \label{eqn:measures-converge}
    \mue
    \tendsto\Vee\volume
  \end{equation}
  in $\dual{\W1p(\mfd)}$ as $\eps\tendsto0$.
\end{prop}
The proof of Proposition~\ref{prop:measures-converge} relies on
results developed in~\cite{gkl}. Let us say that a measure $\mu$ on
$\mfd$ is \emph{$p$-admissible} if the trace $\W1p(\mfd)\to\Lp
p(\mfd,\mu)$ is compact. We then have the following
\begin{lem}[{\cite[Lemma~3.16]{gkl}}]
  \label{lem:psi} 
  Let $(\mfd,\gee)$ be a compact Riemannian manifold, and let $1<p\le
  2$. Let $\xi,\mu$ be $p$-admissible measures. Let us write
  $p'=\frac{p}{p-1}$ for the H\"older conjugate of $p$.
  Then there exists a unique $\Psi\in \W{1}{p'}(\mfd)$ with zero
  $\xi$-average which satisfies the equation
  \begin{equation}
    \label{eqn:psi-eqn}
    \int_\mfd\grad f\cdot\grad\Psi\dv
    =
    \int_\mfd f
    \,\dee{\mu}
    -
    \frac{\mu(\mfd)}{\xi(\mfd)}
    \int_\mfd f
    \,\dee{\xi}
  \end{equation}
  for all $f\in\W1p(\mfd)$. Moreover, we have the following bound on
  the norm of $\Psi$:
  \begin{equation}
    \label{eqn:psi-estimate}
    \norm[\Lp{p'}(\mfd)]{\grad\Psi}
    \le
    (1+\Kay(\mfd,g))
    \mu(\mfd)^{1/p'}
    \norm{T^\mu_p}.
  \end{equation}
\end{lem}
\begin{proof}[Proof of Proposition~\ref{prop:measures-converge}]
  By the monotonicity of the dual spaces of $\W1p(\mfd)$, it suffices to
  prove the desired convergence in $\dual{\W1p(\mfd)}$ for all $1<p<2$.

  Fix $f\in\W1p(\mfd)$. Recalling~\eqref{eqn:dom-dfn} we have that
  \begin{equation}
    \label{eqn:int-estimate-sum}
    \int_\mfd f\Vee\dv
    =
    \int_{\bdom}f\Vee\dv
    +
    \int_{\odom}f\Vee\dv
    +
    \sum_{\site\in\psites}\int_{\cell}f\Vee\dv.
  \end{equation}
  We estimate each piece of the integral individually. Applying the
  generalised H\"older inequality with exponents $(p,p',\infty)$ provide
  us with estimates that show that the first two of these terms vanish
  in the limit:
  \begin{itemize}
  \item
    For the cells near the boundary, observe that since by construction
    $\cell\subseteq\ball{\site}{3\eps}$, the entirety of $\bdom$ is
    contained within a neighbourhood of $\bd\mfd$ with $\bigO[\mfd](\eps)$
    thickness. Therefore we have $\Vol(\bdom)\lesssim_{\mfd,\Vee}\eps$, hence
    \begin{equation}
      \label{eqn:int-estimate-bdom}
      \int_{\bdom}f\Vee\dv
      \le
      \norm[\Lp\infty(\mfd)]{\Vee}
      \norm[\Lp p(\mfd)]{f}
      \Vol(\bdom)^{\frac{1}{p'}}
      \lesssim_{\mfd,\Vee,p}
      \eps^{\frac{p-1}{p}}
      \norm[\Lp p(\mfd)]{f}
    \end{equation}
    which tends to zero as $\eps\tendsto0$.
  \item
    Near $\osites$, since
    $\Vee\in\Cn1(\mfd)$ we have that
    $\norm[\Lp\infty(\odom)]{\Vee}=\bigO[\mfd,\Vee](\eps^\half)$. The
    volume of $\odom$ is finite since $\mfd$ is compact, so
    \begin{equation}
      \label{eqn:int-estimate-zdom}
      \int_{\odom}f\Vee\dv
      \le
      \norm[\Lp\infty(\odom)]{\Vee}
      \norm[\Lp p(\mfd)]{f}
      \Vol(\odom)^{\frac{1}{p'}}
      \lesssim_{\mfd,\Vee,p}
      \eps^\half
      \norm[\Lp p(\mfd)]{f}.
    \end{equation}
    This term, too, tends to zero as $\eps\tendsto0$.
  \end{itemize}
  Thus, to prove~\eqref{eqn:measures-converge}, it remains only to
  show that
  \begin{equation}
    \label{eqn:measures-converge-1}
    \bigabs{
    \sum_{\site\in\psites}\int_{\cell}f\Vee\dv
    -
    \int_{\bd\holes}f\bdparame\dA.
    }\tendsto0
  \end{equation}
  For each $\site\in\psites$, let us define $\Psii\in\W1{p'}(\cell)$ by
  applying Lemma~\ref{lem:psi} on the cell viewed as a Riemannian
  manifold $(\cell,g)$, choosing $\xi=\Vee\volume$ and
  $\mu=\mue$. Moreover rewriting the desired limit in
 ~\eqref{eqn:measures-converge-1} using~\eqref{eqn:psi-eqn}, we then
  have that
  \begin{equation}
    \label{eqn:psi-eqn-sum}
    \begin{aligned}
      \int_{\bd\holes}f\bdparame\dA
      &=
      \sum_{\site\in\psites}
      \int_{\bd\hole}f\bdparame\dA\\
      &=
      \sum_{\site\in\psites}
      \frac{\int_{\bd\hole}\bdparame\dA}{\int_{\cell} \Vee\dv}
      \int_{\cell}f\Vee\dv
      +
      \sum_{\site\in\psites}
      \int_{\cell} \grad f\cdot \grad\Psii\dv.
    \end{aligned}
  \end{equation}
  The first term on the right-hand side of~\eqref{eqn:psi-eqn-sum} tends
  to $\int_\mfd f\Vee\dv$; this follows from
  Proposition~\ref{prop:cell-integral-estimate} together with
 ~\eqref{eqn:int-estimate-sum} through
 ~\eqref{eqn:int-estimate-zdom}. Thus, to prove
 ~\eqref{eqn:measures-converge-1} in turn amounts to showing that the
  second term on the right-hand side of~\eqref{eqn:psi-eqn-sum} vanishes
  in the limit $\eps\tendsto 0$.

  To do this, we seek to bound $\norm[\Lp{p'}(\pcell)] {\grad\Psii}$.
  However, directly applying~\eqref{eqn:psi-estimate} in the present
  setting falls short of producing a sharp enough inequality. To improve
  upon this, we employ an amplification argument.

  Let us again view the cell as a Riemannian manifold, but equip it
  instead with the scaled metric $\gee_\scale$. By considering scaling
  properties we find that $\Psii$ satisfies
  \begin{equation}
    \label{eqn:psi-eqn-scaled}
      \scale^{2-\mfdim}\int_{\cell}\grad[\gee_\scale] f\cdot\grad[\gee_\scale]\Psii\dv[\gee_\scale]
      =
      \scale^{1-\mfdim}\int_{\bd\hole} f\dA[\gee_\scale]
      -
      \frac{\scale^{1-\mfdim}\int_{\bd\hole}\bdparame\dA[\gee_\scale]}{\scale^{-\mfdim}\int_{\cell}\dv}
      \scale^{-\mfdim}\int_{\cell} f\Vee\dv[\gee_\scale].
  \end{equation}
  for all $f\in\W1p(\cell)$. Defining now
  \begin{equation}
    \label{eqn:mmeas-xmeas-dfn}
    \begin{aligned}
      \Mues&=\scale\inv\mues=\scale\inv\bdparame\area[\gee_\scale]\\
      \Xies&=\Vee\volume[\gee_\scale]
    \end{aligned}
  \end{equation}
  we observe that $\Psii$ has zero $\Xies$-average and satisfies an
  equation of the form~\eqref{eqn:psi-eqn} with $\mu=\Mues$ and
  $\xi=\Xies$; as such, the norm of $\Psii$ is controlled by the
  corresponding estimate~\eqref{eqn:psi-estimate}.

  Recalling now the preceding results from this section, we have the
  following:
  \begin{itemize}
  \item
    By construction,
    \begin{equation}
      \label{eqn:mu-estimate}
      \Mues(\cell)^{\frac{1}{p'}}
      \asymp_{\mfd,\Vee,p}
      \left(
      \scale\inv\bdconst
      \left(\scale\radius\right)^{\mfdim-1}
      \right)^{\frac{1}{p'}}
      \asymp_{\mfd,\Vee,p}
      \scale^{\frac{(\mfdim-2)(p-1)}{p}}\eps^{\frac{\mfdim(p-1)}{p}}.
    \end{equation}
  \item
    By Proposition~\ref{prop:p-poincare-constant},
    \begin{equation}
      \label{Kay-estimate}
      \Kay(\cell,\volume[\gee_\scale])\lesssim_{\mfd,p}\scale^p\eps^p.
    \end{equation}
    The best that we can hope happens to the term
    $(1+\Kay(\cell,\volume[\gee_\scale]))$ is that it remains bounded
    in the limit, which is the case as long as
    $\scale\lesssim\eps\inv$.
  \item
    Recall the definition of $T^\mu_p$ from \cite[Corollary
      3.7]{gkl}. Here, with $\mu=\Mues$, this is that
    \begin{equation}
      T^{\Mues}_p:\W1p(\cell,\volume[\gee_\scale])\to\Lp p(\cell,\Mues)
    \end{equation}
    is the linear
    operator which extends the identity on smooth functions. Comparing
    this with the trace
    \begin{equation}
      \trace:\W1p(\pcell,\volume[\gee_\scale])\to\Lp
      p(\bd\hole,\area[\gee_\scale])
    \end{equation}
    which is defined in Proposition~\ref{prop:trace-with-scaling} as
    an extension of the identity on functions on $\cell$ smooth up to
    the boundary of the hole at $\site\in\psites$, we have for
    $f\in\W1p(\pcell)$ that
    \begin{equation}
      \begin{aligned}
        \norm[\Lp p(\cell,\mues)]{T^{\Mues}_p f}
        &\asymp_{\mfd,\Vee,p}
        \scale^{-\frac{1}{p}}
        \bdconst^{\frac{1}{p}}
        \norm[\Lp p(\bd\hole,{\area[\gee_\scale]})]{\trace f}\\
        &\lesssim_{\mfd,\Vee,p}
        \scale^{-\frac{1}{p}}
        \bdconst^{\frac{1}{p}}
        \max\left\{
        \scale^{-\frac{1}{p}}
        \bdconst^{-\frac{1}{p}}
        ,
        \scale^{\frac{p-1}{p}}
        \bdconst^{-\frac{p-1}{(\mfdim-1)p}}
        \eps^{\frac{\mfdim(p-1)}{(\mfdim-1)p}}
        \right\}
        \norm[\W1p(\pcell,{\volume[\gee_\scale]})]{f}
      \end{aligned}
    \end{equation}
    using~\eqref{eqn:trace-with-scaling}. By replacing
    $\norm[\W1p(\pcell,{\volume[\gee_\scale]})]{f}$ with the
    larger quantity $\norm[\W1p(\cell,{\volume[\gee_\scale]})]{f}$, we
    conclude that
    \begin{equation}
      \label{eqn:T-estimate}
      \norm{T^{\Mues}_p}\lesssim_{\mfd,\Vee,p}
      \scale^{-\frac{1}{p}}
      \bdconst^{\frac{1}{p}}
      \max\left\{
      \scale^{-\frac{1}{p}}
      \bdconst^{-\frac{1}{p}}
      ,
      \scale^{\frac{p-1}{p}}
      \bdconst^{-\frac{p-1}{p(\mfdim-1)}}
      \eps^{\frac{\mfdim(p-1)}{(\mfdim-1)p}}
      \right\}.
    \end{equation}
  \end{itemize}
  Inserting the
  estimates~\eqref{eqn:mu-estimate}--\eqref{eqn:T-estimate}
  into~\eqref{eqn:psi-estimate}, we recover
  \begin{equation}
    \label{eqn:psi-rate}
    \norm[\Lp{p'}(\cell,\gee_\scale)]{\grad[\gee_\scale]\Psii}
    \lesssim_{\mfd,\Vee,p}
    \rate(\scale,\eps)
  \end{equation}
  in which, for brevity, we have defined
  \begin{equation}
    \label{eqn:rate-dfn}
    \rate(\scale,\eps)
    =
    \max\Big\{
    \scale^{\mfdim-2-\frac{\mfdim}{p}}
    \bdconst^{\frac{\mfdim(p-1)}{p}}
    \eps^{\frac{\mfdim(p-1)}{p}}
    ,
    \scale^{\mfdim-1-\frac{\mfdim}{p}}
    \bdconst^{\frac{\mfdim-p}{(\mfdim-1)p}}
    \eps^{\frac{\mfdim^2(p-1)}{(\mfdim-1)p}}
    \Big\}.
  \end{equation}
  We return now to~\eqref{eqn:psi-eqn-sum}. Recall that for any family
  $\sequence{a_i}{i\in I}$ of real numbers indexed by a finite set
  $I$,
  \begin{equation}
    \label{eqn:holder-sum}
    \sum_{i\in I}\abs{a_i}
    \le
    (\card{I})^{\frac{1}{p'}}\left(\sum_{i\in I}\abs{a_i}^p\right)^{\frac{1}{p}}
  \end{equation}
  which follows from an application of H\"older's inequality. Using
  this, we obtain the following:
  \begin{equation}
    \label{eqn:holder-sum-2}
    \begin{aligned}
      \sum_{\site\in\psites}
      \int_{\cell}\grad f\cdot\grad \Psii\dv
      &=
      \scale^{2-\mfdim}
      \sum_{\site\in\psites}
      \int_{\cell}\grad[\gee_\scale] f\cdot\grad[\gee_\scale] \Psii\dv[\gee_\scale]\\
      &\le
      \scale^{2-\mfdim}
      \sum_{\site\in\psites}
      \norm[\Lp{p}(\cell,{\volume[\gee_\scale]})]{\grad[\gee_\scale] f}
      \norm[\Lp{p'}(\cell,{\volume[\gee_\scale]})]{\grad[\gee_\scale]\Psii}\\
      &\lesssim_{\mfd,\Vee,p}
      \scale^{2-\mfdim}
      \rate(\scale,\eps)
      \abs{\psites}^{\frac{1}{p'}}
      \left(
      \sum_{\site\in\psites}
      \norm[\Lp p(\cell,{\volume[\gee_\scale]})]{\grad[\gee_\scale] f}^p
      \right)^{\frac{1}{p}}\\
      &\lesssim_{\mfd,\Vee,p}
      \scale^{1-\mfdim+\frac{\mfdim}{p}}
      \eps^{\frac{-\mfdim(p-1)}{p}}
      \rate(\scale,\eps)\norm[\Lp p(\mfd)]{\grad f}\\
      &\lesssim_{\mfd,\Vee,p}
      \max\left\{
      \scale\inv,
      \bdconst^{\frac{\mfdim-p}{(\mfdim-1)p}}
      \eps^{\frac{\mfdim(p-1)}{(\mfdim-1)p}}
      \right\}\norm[\Lp p(\mfd)]{\grad f}
    \end{aligned}
  \end{equation}
  where we arrived on the penultimate line by observing that
  $\card{\psites}\le\card{\sites}\asymp_\mfd\eps^{-\mfdim}$,
  and also that
  \begin{equation}
    \sum_{\site\in\psites}
    \norm[\Lp p(\cell,{\volume[\gee_\scale]})]{\grad[\gee_\scale] f}^p
    \le\norm[\Lp p(\mfd,{\volume[\gee_\scale]})]{\grad[\gee_\scale]f}^p
    =\scale^{\mfdim-p}\norm[\Lp p(\mfd)]{\grad f}^p
  \end{equation}
  by scaling properties of the $\Lp p$-norm together with
 ~\eqref{eqn:holder-sum}.
  
  The sharpest bound we can recover from~\eqref{eqn:holder-sum-2} is
  achieved by balancing the two branches of the maximum with the
  choice $\scale=\bdconst^{-\frac{\mfdim-p}{(\mfdim-1)p}}
  \eps^{-\frac{\mfdim(p-1)}{(\mfdim-1)p}}$, which gives us that
  \begin{equation}
    \label{eqn:grad-f-grad-psi}
    \sum_{\site\in\psites}
    \int_{\cell} \grad f\cdot \grad\Psii\dv
    \lesssim_{\mfd,\Vee,p}
    \bdconst^{-\frac{\mfdim-p}{(\mfdim-1)p}}
    \eps^{-\frac{\mfdim(p-1)}{(\mfdim-1)p}}
    \norm[\Lp p(\mfd)]{\grad f}.
  \end{equation}
  For $d>p>1$, this vanishes as $\eps\tendsto0$ so long as
  \begin{equation}
    \bdconst=\littleo[\mfd,\Vee,p]\left(\eps^{-\frac{\mfdim(p-1)}{\mfdim-p}}\right),
  \end{equation}
  which is hypothesis~\eqref{eqn:bdconst-range}. Thus we arrive
  at the desired conclusion that the corresponding term in
  ~\eqref{eqn:psi-eqn-sum} vanishes in the limit, which finishes the
  proof.
\end{proof}
\subsection{Approaching the Schr\"odinger problem in the limit}
\label{ssec:problem-in-limit}
Equipped with the results from the previous section, we proceed now
towards showing that the spectrum of the Robin problem on the
perforated domains $\dome$ converge to that of the Schr\"odinger
problem on $\mfd$. This is Proposition~\ref{prop:main}, from
which the main results Theorem~\ref{thm:main-nonneg}
and~\ref{thm:main} are recovered.
Because it is useful later, and for the clarity of isolating the
properties that are essential for carrying out the proof, we formulate
the results in this section in the following setting:

Let a subset $T$ of a $\mfdim$-dimensional Riemannian manifold $\mfd$
be such that $\cl{T}\subseteq\interior{\mfd}$, and let
$\dom=\mfd\setminus T$.  We consider the following eigenvalue problem,
formulated in weak form on the space $\Wz12(\dom)$: given a
measure~$\nu$ on $\dom$, we seek non-trivial $\lam[]{\Rch}$ and
$\yu\in\Wz12(\dom)$ satisfying for all $v\in\Wz12(\dom)$ that
\begin{equation}
  \label{eqn:weak-intermediate}
  \int_{\dom}\grad \yu\cdot\grad v\dv
  +
  \int_{\dom}\yu v\,\dee\nu
  +
  \int_{\bdmfd}\yu v\Bdparam\dA
  =
  \lam[]{\Rch}
  \int_{\dom}\yu v\dv.
\end{equation}
If $\mfdim\ge3$ and integration against $\nu$ defines a linear
functional in $\dual{\W1{\frac{\mfdim}{\mfdim-1}}(\mfd)}$, then we
again have a spectrum consisting of discrete eigenvalues; when
$\mfdim=2$, it suffices to have
$\nu\in\dual{\W1{2,-\half}(\mfd)}$. In this case we write
$\lam{\Rch}(\dom,\nu,\Bdparam)$---without a superscript---for the $k$-th eigenvalue, and
$\yuk{\Rch}(\dom,\nu,\Bdparam)\in\Wz12(\dom)$ for the corresponding
eigenfunctions, chosen to be orthonormal in $\Lp2(\dom)$.

We then have the following proposition:
\begin{prop}[Conditions for spectrum convergence]
  \label{prop:main-intermediate}
  Let $\mfd$ be a Riemannian manifold of dimension $\mfdim\ge2$, and
  let $\Vee$ be an admissible potential on $\mfd$. Let
  $\sequence{\dome[\eps]}{\eps>0}$ be a family of domains in $\mfd$,
  and let $\sequence{\nue}{\eps>0}$ be a family of Radon measures
  such that $\nue$ is supported on $\cl{\dome[\eps]}$. Let the
  following conditions be satisfied:
  \begin{itemize}
  \item[{\hypone}] For all $\eps>0$,
    $\dome[\eps]=\mfd\setminus\holes[\eps]$ where
    $\cl{\holes[\eps]}\subseteq\interior{\mfd}$, and
    $\Vol(\holes[\eps])\tendsto0$ as $\eps\tendsto0$.
  \item[{\hyptwo}] There exist an equibounded family $\sequence{\gext[\eps]}{\eps>0}$ of
    extension operators $\gext[\eps]:\Wz12(\dome[\eps])\to\Wz12(\mfd)$,
  \end{itemize}
  and, writing $\Wprod$ for $\W1{\frac{\mfdim}{\mfdim-1}}(\mfd)$ if
  $\mfdim\ge3$ and for $\W1{2,-\half}(\mfd)$ if $\mfdim=2$,
  \begin{itemize}
  \item[{\hypthree}] $\nue\in\dual{\Wprod}$ all $\eps>0$, and $\nue\tendsto\Vee$ in
    $\dual{\Wprod}$ as $\eps\tendsto0$.
  \end{itemize}
  Then there is a sequence $\sequence{\eps_n}{n\in\N}$ such that
  $\eps_n\tendsto0$ as $n\tendsto\infty$, along which
  \begin{itemize}
  \item $\lam{\Rch}(\dome[\eps_n],\nue,\Bdparam)\tendsto\lam{\Sch}(\mfd,\Vee,\Bdparam)$, and
  \item
    $\gext[\eps]\yuk{\Rch}(\dome[\eps_n],\nue[\eps_n],\Bdparam)$
    tends to $\yuk{\Sch}(\mfd,\Vee,\Bdparam)$ in the weak topology
    on $\W12(\mfd)$.
  \end{itemize}
\end{prop}
\begin{rmk}
  By the embeddings~\eqref{eqn:orlicz-sobolev-embedding}, for any
  $\mfdim\ge2$, {\hypthree} is satisfied as soon as
  $\nue\tendsto\Vee$ in $\dual{\W1p(\mfd)}$ for some
  $p<\frac{\mfdim}{\mfdim-1}$.
\end{rmk}
\begin{rmk}
  Hypotheses {\hypone}, {\hyptwo}, and {\hypthree} make precise the
  two conditions considered in Section~\ref{ssec:outline-of-proof}.
  Compare also to conditions \hyp{M1} through \hyp{M3} in Girouard,
  Karpukhin, and Lagac\'e~\cite{gkl}, which support a similar
  conclusion for variational eigenvalues.
\end{rmk}
Let us first see that Proposition~\ref{prop:main} is recovered from
Proposition~\ref{prop:main-intermediate}.
\begin{proof}[Proof of Proposition~\ref{prop:main} from Proposition~\ref{prop:main-intermediate}]
  Observe that with the choices $\dom=\dome[\eps]$ and
  $\nu=\mue=\bdparame\area$, \eqref{eqn:weak-intermediate} becomes
  identical to \eqref{eqn:weak-robin}; in particular,
  $\lam{\Rob}(\dome,\bdparame,\Bdparam)=\lam{\Rch}(\dome,\mue,\Bdparam)$. Thus,
  the proof amounts to verifying {\hypone}, {\hyptwo}, and
  {\hypthree}.

By hypothesis, $\alpha=\alpha(\eps)$
satisfies~\eqref{eqn:bdconst-hyp}. We thus have the following:
\begin{itemize}
  \item Using \eqref{eqn:radius-dfn}, we estimate the total volume of
    the holes:
  \begin{equation}
    \Vol(\holes)=
    \sum_{\site\in\psites}\Vol{\hole}
    \lesssim_{\mfd,\Vee}
    \underbrace{
      \card{\psites}
      }_{\lesssim{\eps^{-\mfdim}}}
    \left(
    \bdconst^{-\frac{1}{\mfdim-1}}
    \eps^{\frac{\mfdim}{\mfdim-1}}
    \right)^{\mfdim}
    \lesssim_{\mfd,\Vee}
    \left(\frac{\eps}{\bdconst}\right)^{\frac{\mfdim}{\mfdim-1}}.
    \end{equation}
since $\bdconst\inv\in\littleo(\eps\inv)$, this vanishes in the limit $\eps\tendsto0$, and {\hypone} is verified.
\item {\hyptwo} is exactly Lemma~\ref{lem:extension-operator}.
\item {\hypthree} follows from  Proposition~\ref{prop:measures-converge} and the embeddings~\eqref{eqn:orlicz-sobolev-embedding}.
\end{itemize}
This finishes the proof.
\end{proof}
The goal for the remainder of this section is the proof of
Proposition~\ref{prop:main-intermediate}. This proceeds in two
steps. First, we again restrict our attention to the case of
continuously differentiable potentials $\Vee\in\Cn1(\mfd)$, and show
that we can attain the desired convergence in this case; this is done
in the first half of this section up to
Proposition~\ref{prop:orthonormality}. We then show that the
Schr\"odinger eigenvalues and eigenfunctions converge in the
appropriate sense, so that the main result is recovered via the result
for $\Cn1$ potentials, combined with density of $\Cn1(\mfd)$ in the
space of admissible potentials.

Let us now assume $\Vee\in\Cn1(\mfd)$. Throughout the remainder of
this section, $\sequence{\dome[\eps]}{\eps>0}$ shall always denote a
family of domains in a Riemannian manifold $\mfd$ of dimension
$\mfdim\ge2$, and $\sequence{\nue}{\eps>0}$ a family of Radon measures
such that each $\nue$ is supported on $\cl{\dome[\eps]}$.

We begin by establishing upper semi-continuity of eigenvalues
$\lam{\Rch}(\dome[\eps],\nue,\Bdparam)$. The proof proceeds similarly
to \cite[Proposition 4.7]{gkl}; see also \cite[Proposition
  1.1]{kokarev}.
\begin{prop}[Upper semi-continuity of eigenvalues]
  \label{prop:robin-eigenvalues-upper-semicontinuous}
  Let $\sequence{\dome[\eps]}{\eps>0}$ and $\sequence{\nue}{\eps>0}$
  satisfy hypotheses {\hypone} and {\hypthree}. Then for each
  $k\in\Npos$, we have that
  \begin{equation}
    \limsup_{\eps\tendsto0}
  \lam{\Rch}(\dome[\eps],\nue,\Bdparam)\le\lam{\Sch}(\mfd,\Vee,\Bdparam).
  \end{equation}
\end{prop}
\begin{proof}
  We do this by showing that, for any $\delta>0$, we have for all
  sufficiently small $\eps>0$ that
  \begin{equation}
    \label{eqn:robin-eigenvalue-limsup}
    \lam{\Rch}(\dome[\eps],\nue,\Bdparam)
    <
    \lam{\Sch}(\mfd,\Vee,\Bdparam)+\delta.
    \end{equation}
  Fix $\delta>0$. Using the variational characterisation of the
  Schr\"odinger eigenvalues, we choose a $k$-dimensional subspace
  $F_k$ of $\Cn\infty(\mfd)\cap\Wz12(\mfd)$ such that
  \begin{equation}
    \label{eqn:schroedinger-minimax-est}
    \sup_{f\in F_k\setminus\{0\}}\Rayleigh{\Sch}(\mfd,\Vee,\Bdparam, f)
    <
    \lam{\Sch}(\mfd,\Vee,\Bdparam)
    +\thalf\delta.
  \end{equation}
  Now the Rayleigh quotient corresponding to the
  problem~\eqref{eqn:weak-intermediate} is
  \begin{equation}
    \label{eqn:intermediate-rayleigh}
  \Rayleigh{\Rch}(\dom,\nu,\Bdparam, f)
  =
  \frac{\int_\dom\abs{\grad f}^2\dv+\int_\dom f^2\,\dee\nu+\int_{\bdmfd}f^2\Bdparam\dA}{\int_\dom f^2\dv}.
  \end{equation}
  Comparing this with the Rayleigh
  quotient~\eqref{eqn:schroedinger-rayleigh} for the Schr\"odinger
  problem and using the convergences {\hypone} and {\hypthree}, we
  have for any $f\in\Cn\infty(\mfd)$ that
  \begin{equation}
    \label{eqn:rayleigh-converge}
    \lim_{\eps\tendsto0}
    \Rayleigh{\Rch}(\dome[\eps],\nue,\Bdparam,f)
    =\Rayleigh{\Sch}(\mfd,\Vee,\Bdparam,f).
    \end{equation}
  For each $\eps>0$, the map
  $f\mapsto\Rayleigh{\Rch}(\dome[\eps],\nue,\Bdparam,f)$ is continuous
  on $F_k\setminus\{0\}$, and in particular on the set
  \begin{equation}
    \overline{F}_k=\{f\in
    F_k
    \,:
    \norm[\Lp2(\mfd)]{f}=1\}\subseteq\Cn\infty(\mfd),
  \end{equation}
  which is compact since $F_k$ is finite-dimensional. Therefore the
  convergence~\eqref{eqn:rayleigh-converge} is uniform over
  $\overline{F}_k$, consequently
  \begin{equation}
    \label{eqn:rayleigh-converge-2}
    \begin{aligned}
      \lim_{\eps\tendsto0}
      \sup_{f\in F_k\setminus\{0\}}
      \Rayleigh{\Rch}(\dome[\eps],\nue,\Bdparam,f)
      &=
      \lim_{\eps\tendsto0}
      \sup_{f\in \overline{F}_k}
      \Rayleigh{\Rch}(\dome[\eps],\nue,\Bdparam,f)\\
      &=
      \sup_{f\in \overline{F}_k}
      \lim_{\eps\tendsto0}
      \Rayleigh{\Rch}(\dome[\eps],\nue,\Bdparam,f)\\
      &=
      \sup_{f\in\overline{F}_k}\Rayleigh{\Sch}(\mfd,\Vee,\Bdparam,f)\\
      &=
      \sup_{f\in F_k\setminus\{0\}}\Rayleigh{\Sch}(\mfd,\Vee,\Bdparam,f).
      \end{aligned}
    \end{equation}
  Finally, it follows from {\hypone} that for all sufficiently small
  $\eps>0$ the subspace $F_k$ restricts to a subspace of
  $\Wz12(\dome)$ that is still $k$-dimensional. Using the variational
  characterisation of $\lam{\Rch}(\dome[\eps],\nue,\Bdparam)$ followed
  by~\eqref{eqn:rayleigh-converge-2}, we then have for all
  sufficiently small $\eps>0$ that
  \begin{equation}
    \label{eqn:robin-minimax-est}
      \lam{\Rch}(\dome[\eps],\nue,\Bdparam)
      \le
    \sup_{f\in F_k\setminus\{0\}}
    \Rayleigh{\Rch}(\dome[\eps],\nue,\Bdparam, f)
    \le
   \sup_{f\in F_k\setminus\{0\}}
   \Rayleigh{\Sch}(\mfd,\Vee,\Bdparam,f)
   +\thalf\delta.
  \end{equation}
Combining~\eqref{eqn:schroedinger-minimax-est}
and~\eqref{eqn:robin-minimax-est}, we
conclude~\eqref{eqn:robin-eigenvalue-limsup} as desired.
\end{proof}
As an immediate corollary, we have
\begin{cor}[Uniform boundedness of eigenvalues]
  \label{cor:robin-eigenvalues-bounded}
  Let $\sequence{\dome[\eps]}{\eps>0}$ and
    $\sequence{\nue}{\eps>0}$ satisfy hypotheses {\hypone} and
    {\hypthree}. Then for each $k\in\Npos$,
  there exist constants $\eps_0>0$ and $C_{\mfd,\Vee}>0$ such that for
  all $0<\eps<\eps_0$, we have that
  $\abs{\lam{\Rch}(\dome[\eps],\nue,\Bdparam)}<C_{\mfd,\Vee}$.
\end{cor}
The next proposition establishes that the eigenfunctions
$\sequence{\yuek[\eps]}{\eps>0}$, which we recall are elements of
$\Wz12(\dome[\eps])$ with unit norm in $\Lp2(\dome[\eps])$, can be
extended to a family of functions in $\Wz12(\mfd)$ in a way that is
uniformly bounded for all sufficiently small $\eps>0$. It is from this
family that will be extracted a subsequence in $\Wz12(\mfd)$
converging to a Schr\"odinger eigenfunction.
\begin{prop}[Uniformly bounded extensions of eigenfunctions]
  \label{prop:norm-bounded}
  Let $\sequence{\dome[\eps]}{\eps>0}$ and $\sequence{\nue}{\eps>0}$
  satisfy hypotheses {\hypone}, {\hyptwo}, and {\hypthree}. Then there
  exist constants $\eps_0>0$ and $C_{\mfd,\Vee}>0$ such that for all
  $0<\eps<\eps_0$, we have that
  $\norm[\Wz12(\mfd)]{\gext[\eps]\yuek[\eps]}<C_{\mfd,\Vee}$
\end{prop}
\begin{proof}
We have by {\hyptwo} the extension operators
$\gext[\eps]:\Wz12(\dome[\eps])\to\Wz12(\mfd)$ which are uniformly
bounded for small $\eps>0$; to show the uniform boundedness of
$\sequence{\gext[\eps]\yuek[\eps]}{\eps>0}$ it therefore suffices to
show that the eigenvalues $\yuek[\eps]\in\Wz12(\dome[\eps])$ are in
fact also bounded uniformly in $\W12(\mfd)$.
For elements lying in $\Wz12(\mfd)$, the $\Wz12(\mfd)$-norm coincides
with the norm inherited from the larger space $\W12(\mfd)$; we compute
with the latter.  Choosing $v=\yuek[\eps]$ in the weak
formulation~\eqref{eqn:weak-intermediate}, we rewrite
  \begin{equation}
    \label{eqn:W12-norm-2}
    \begin{aligned}
      \norm[\Lp2({\dome[\eps]})]{\grad\yuek[\eps]}^2
      &=
      \lam{\Rch}(\dome[\eps],\nue,\Bdparam)\norm[\Lp2({\dome[\eps]})]{\yuek[\eps]}^2
      -
      \int_{\dome[\eps]} (\yuek[\eps])^2\dnue
      -
      \int_{\bdmfd} (\yuek[\eps])^2\Bdparam\dA
      \\
      &=
      \lam{\Rch}(\dome[\eps],\nue,\Bdparam)\norm[\Lp2({\dome[\eps]})]{\yuek[\eps]}^2
      -
      \left(
      \int_{\dome[\eps]}(\yuek[\eps])^2\dnue
      -
      \int_\mfd(\gext[\eps]\yuek[\eps])^2\Vee\dv\right)\\
      &\quad-
      \int_{\bdmfd}(\yuek[\eps])^2\Bdparam\dA
      -
      \int_\mfd(\gext[\eps]\yuek[\eps])^2\Vee\dv\\
      &=
      \lam{\Rch}(\dome[\eps],\nue,\Bdparam)\norm[\Lp2({\dome[\eps]})]{\yuek[\eps]}^2
      -
      \left(
      \int_\mfd(\gext[\eps]\yuek[\eps])^2\dnue
      -
      \int_\mfd(\gext[\eps]\yuek[\eps])^2\Vee\dv
      \right)
      \\
      &\quad-
      \int_{\bdmfd}(\yuek[\eps])^2\Bdparam\dA
      -
      \int_{\dome[\eps]}(\yuek[\eps])^2\Vee\dv
      -
      \int_{\holes[\eps]}(\gext[\eps]\yuek[\eps])^2\Vee\dv.
    \end{aligned}
  \end{equation}
  We estimate each of the terms in this sum individually:
  \begin{itemize}
  \item
    $\lam{\Rch}(\dome[\eps],\nue,\Bdparam)\norm[\Lp2({\dome[\eps]})]{\yuek[\eps]}^2\le
    C_{\mfd,\Vee}$, since the eigenfunctions $\yuek[\eps]$ are
    normalised in $\Lp2(\dome)$, and the eigenvalues are bounded by
    Corollary~\ref{cor:robin-eigenvalues-bounded}.
  \item
    When $\mfdim\ge3$, we have that $(\gext[\eps]\yuek[\eps])^2$ is in
    $\W1{\frac{\mfdim}{\mfdim-1}}(\mfd)$ by
    Proposition~\ref{prop:sobolev-multiplication-i}; when $\mfdim=2$
    we have that it is in $\W1{2,-\half}(\mfd)$ by
    Proposition~\ref{prop:sobolev-multiplication-ii}. In either case,
    the term in parentheses vanishes by {\hypthree}.
  \item
    $\int_{\bdmfd}
    (\yuek[\eps])^2\Bdparam\dA\lesssim\norm[\Lp\infty(\bdmfd)]{\Bdparam}$,
    by the trace theorem for bounded domains with Lipschitz
    boundary.
  \item     
    Choose $q$ such that $2\le\mfdim<q<\infty$. This implies that
    $\frac{2q}{q-2}\le\frac{2d}{d-2}$, so we have the Sobolev
    embedding $\W12(\mfd)\to\Lp{\frac{2q}{q-2}}$; applying the
    generalised H\"older inequality with exponents
    $(\frac{2q}{q-2},\frac{2q}{q-2},\frac{q}{2})$ yields
    \begin{equation}
      \label{eqn:W12-norm-3}
      \begin{aligned}
        \int_{\holes[\eps]}(\gext[\eps]\yuek[\eps])^2\Vee\dv
        &\le
        \norm[\Lp{\frac{2q}{q-2}}(\mfd)]{\gext[\eps]\yuek[\eps]}^2
        \norm[\Lp{\frac{q}{2}}(\mfd)]{\Vee\chf{\holes[\eps]}}\\
        &\lesssim_{\mfd}
        \norm{\gext[\eps]}^2\norm[\W{1}{2}({\dome[\eps]})]{\yuek[\eps]}^2
        \norm[\Lp\infty(\mfd)]{\Vee}
        \Vol(\holes[\eps])^{\frac{2}{q}}
      \end{aligned}
    \end{equation}
    which vanishes in the limit $\eps\tendsto0$ because $\Vol(\holes[\eps])$ does, by {\hypone}.
  \end{itemize}
  substituting these estimates into~\eqref{eqn:W12-norm-2} and taking
  the limit $\eps\tendsto0$, we conclude that
  $\norm[\W{1}{2}(\mfd)]{\gext[\eps]\yuek[\eps]}$ is bounded
  uniformly over all sufficiently small $\eps>0$, as desired.
\end{proof}
From this point on, let us write
$\Yuek[\eps]=\gext[\eps]\yuek[\eps]\in\W12(\mfd)$. By the
Banach--Alaoglu theorem, we can extract a weakly convergent sequence
from the norm-bounded family
$\sequence{\yuek[\eps]}{\eps\in(0,\eps_0)}$ along which
$\eps\tendsto0$.  Using Corollary~\ref{cor:robin-eigenvalues-bounded},
we can furthermore replace this sequence with a subsequence along
which the corresponding Robin eigenvalues also converge. That is, we
have deduced
\begin{cor}[Weak convergence of eigenfunctions]
  \label{cor:weak-convergence}
  Let $\sequence{\dome[\eps]}{\eps>0}$ and $\sequence{\nue}{\eps>0}$
  satisfy hypotheses {\hypone}, {\hyptwo}, and {\hypthree}. Then for
  each $k\in\Npos$, there exists a sequence chosen from
  $\sequence{\yuek[\eps]}{\eps>0}$ along which $\eps\tendsto0$, such
  that
  \begin{itemize}
  \item
    $\Yuek[\eps]\tendsto\Yuuk$ in the weak topology on $\Wz12(\mfd)$,
    and
  \item
    $\lam{\Rch}(\dome[\eps],\nue,\Bdparam)\tendsto\laam$ in $\R$
  \end{itemize}
  for some $\Yuuk\in\Wz12(\mfd)$ and $\laam\in\R$.
\end{cor}
Henceforth when we speak of taking the limit $\eps\tendsto0$, let it
always be along the sequence whose existence is established by
Corollary~\ref{cor:weak-convergence}.
We next show that the weak limit $\Yuuk\in\Wz12(\mfd)$ is in fact a
Schr\"odinger eigenfunction with corresponding eigenvalue $\laam$.
\begin{prop}[The equation satisfied in the limit]
  \label{prop:equation-in-limit}
  The weak limit $\Yuuk$ of eigenfunctions and the limit $\laam$ of
  eigenvalues satisfy~\eqref{eqn:weak-schroedinger}:
  that is, for all $v\in\Wz12(\mfd)$, we have that
  \begin{equation}
    \label{eqn:equation-in-limit}
    \int_{\mfd}{\grad\Yuuk}\cdot{\grad v}
    +
    \int_\mfd \Vee\Yuuk v\dv
    +
    \int_{\bdmfd}\Yuuk v\Bdparam\dA
    =
    \laam
    \int_{\mfd}\Yuuk v.
  \end{equation}
\end{prop}
\begin{proof}
  It suffices to show this for $v\in\Cn{\infty}(\mfd)$.  Using that
  $\yuek[\eps]$ is a solution to~\eqref{eqn:weak-intermediate}, we
  begin by rewriting
  \begin{equation}
    \label{eqn:equation-in-limit-1}
    \begin{aligned}
      \int_{\mfd}\grad\Yuek[\eps]\cdot\grad v\dv
        &=
        \int_{\dome[\eps]}\grad\Yuek[\eps]\cdot\grad v\dv
        +
        \int_{\holes[\eps]}\grad\Yuek[\eps]\cdot\grad v\dv
      \\
      &=
      \left(
      \lam{\Rch}(\dome[\eps],\nue,\Bdparam)
      \int_{\dome[\eps]}\Yuek[\eps]v\dv
      -
      \int_{\dome[\eps]}\Yuek[\eps]v\dnue
      -
      \int_{\bdmfd}\yuek[\eps]v\Bdparam\dA
      \right)\\
      &\quad+
        \int_{\holes[\eps]}\grad\Yuek[\eps]\cdot\grad v\dv\\
      &=
        \lam{\Rch}(\dome[\eps],\nue,\Bdparam)
        \int_{\mfd}\Yuek[\eps]v\dv
      -
      \lam{\Rch}(\dome[\eps],\nue,\Bdparam)
      \int_{\holes[\eps]}\Yuek[\eps]v\dv\\
      &\quad-
      \int_{\dome[\eps]}\Yuek[\eps]\dnue
      -
      \int_{\bdmfd}\yuek[\eps] v\Bdparam\dA
      +
        \int_{\holes[\eps]}\grad\Yuek[\eps]\cdot\grad v\dv
    \end{aligned}
  \end{equation}
  in which, in the limit as $\eps\tendsto0$, we have the following:
  \begin{itemize}
  \item
    $\lam{\Rch}(\dome[\eps],\nue,\Bdparam)$ is bounded, uniformly over
    small $\eps>0$, by
    Proposition~\ref{cor:robin-eigenvalues-bounded}.
    \item
    An application of the generalised H\"older
  inequality gives us that
  \begin{equation}
    \label{eqn:holder-3-1}
    \int_{\holes[\eps]}\Yuek[\eps]v\dv
      \le
      \norm[\Lp 2(\mfd)]{\yuek[\eps]}
      \norm[\Lp \infty(\mfd)]{v}
      \Vol(\holes[\eps])^\half
  \end{equation}
  which vanishes in the limit, because
  $\norm[\Lp2(\mfd)]{\Yuek[\eps]}\le\norm[\W12(\mfd)]{\Yuek[\eps]}$ is
  bounded by Proposition~\ref{prop:norm-bounded} and
  $\Vol(\holes[\eps])\tendsto0$ by {\hypone}.
  \item
  A computation identical to~\eqref{eqn:holder-3-1} except with
  $\grad\yuek[\eps]$ and $\grad v$ in place of $\yuek[\eps]$ and $v$
  gives us also that
  \begin{equation}
    \label{eqn:holder-3-2}
    \int_{\holes[\eps]}\grad \Yuek[\eps]\cdot\grad v\dv\tendsto0.
  \end{equation}
  \item
    From the weak convergence~$\Yuek[\eps]\to\Yuuk$, which is
    Corollary~\ref{cor:weak-convergence}, and the convergence of
    measures $\nue\tendsto\Vee\volume$, which is {\hypthree}, we
    deduce
  \begin{equation}
    \label{eqn:equation-in-limit-measures}
    \int_{\dome[\eps]}\yuek[\eps] v\dnue
    \tendsto
    \int_\mfd\Yuuk v\Vee\dv.
  \end{equation}
  \item
  The weak convergence $\yuek[\eps]\to\Yuuk$ in $\W12(\mfd)$ implies
  strong convergence in $\Lp2(\mfd)$, hence
  \begin{equation}
    \int_\mfd\Yuek[\eps]v\dv\tendsto\int_\mfd\Yuuk[\eps]v\dv.
  \end{equation}
\end{itemize}
  Substituting all of this
  into~\eqref{eqn:equation-in-limit-1}, we
  recover~\eqref{eqn:equation-in-limit} as desired.
\end{proof}
Proposition~\ref{prop:equation-in-limit} shows that the limit $\laam$
of eigenvalues is a Schr\"odinger eigenvalue with corresponding
eigenfunction $\Yuuk.$ Next, we show moreover that orthonormality is
preserved in the limit, that is, that $\sequence{\Yuuk}{k\in\Npos}$ is
an orthonormal family in $\Lp2(\mfd)$. A consequence of this is that
the subspace of $\W12(\mfd)$ spanned by $\sequence{\Yuuk[i]}{i=1\dots
  k}$ is $k$-dimensional for each $k\in\Npos$; an induction on
dimension in the variational characterisaion then shows that
\emph{every} Schr\"odinger eigenvalue occurs as a limit in this way,
that is, that
$\lam{\Rch}(\dome[\eps],\nue,\Bdparam)=\lam{\Sch}(\mfd,\Vee,\Bdparam)$
for each $k\in\Npos$, and that (up to relabelling eigenfunctions for
the same eigenvalue with multiplicity greater than one) $\Yuuk$ is the
$k$-th Schr\"odinger eigenfunction.
\begin{prop}[Orthonormality in the limit]
  \label{prop:orthonormality}
    Let $\sequence{\dome[\eps]}{\eps>0}$ and
    $\sequence{\nue}{\eps>0}$ satisfy hypotheses {\hypone}, {\hyptwo}, and
    {\hypthree}. Then
  \begin{equation}
    \label{eqn:orthonormality}
    \idot[\Lp2(\mfd)]{\Yuuk[i]}{\Yuuk[j]}
    =
    \lim_{\eps\tendsto0}\idot[\Lp2(\mfd)]{\Yue[\eps]_i}{\Yue[\eps]_j}
    =
    \delta_{ij}
  \end{equation}
  in which $\delta_{ij}$ is the Kronecker delta.
\end{prop}
\begin{proof}
  For $k\in\Npos$, $\Yuek[\eps]\tendsto\Yuuk$ weakly in $\W12(\mfd)$
  and therefore strongly in $\Lp2(\mfd)$, by
  Proposition~\ref{cor:weak-convergence}. From this follows the first
  equality in~\eqref{eqn:orthonormality}, and it remains to prove the
  second.

  For $i,j\in\Npos$, we have that
  \begin{equation}
    \label{eqn:idot}
    \idot[\Lp2(\mfd)]{\Yue[\eps]_i}{\Yue[\eps]_j}
    =
    \int_{\dome[\eps]}
    \Yue[\eps]_i\Yue[\eps]_j\dv
    +
    \int_{\holes[\eps]}
    \Yue[\eps]_i\Yue[\eps]_j\dv
    =
    \delta_{ij}
    +
    \int_{\holes[\eps]}
    \Yue[\eps]_i\Yue[\eps]_j\dv
  \end{equation}
  using the orthonormality of the Robin eigenfunctions $\yue[\eps]_i$,
  $\yue[\eps]_j$. To see that the remaining integral over
  $\holes[\eps]$ vanishes in the limit, choose $\mfdim<q<\infty$;
  applying the generalised H\"older's inequality in the same way as
  in~\eqref{eqn:W12-norm-3} yields
  \begin{equation}
    \label{eqn:idot-2}
    \begin{aligned}
      \int_{\holes}
      \Yue[\eps]_i\Yue[\eps]_j
      \dv
      &\le
      \norm[\Lp{\frac{2q}{q-2}}(\mfd)]{\Yue[\eps]_i}
      \norm[\Lp{\frac{2q}{q-2}}(\mfd)]{\Yue[\eps]_i}
      \norm[\Lp{\frac{q}{2}}(\mfd)]{\chf{\holes[\eps]}}\\
      &\lesssim_\mfd
      \norm[\W12(\mfd)]{\Yue[\eps]_i}
      \norm[\W12(\mfd)]{\Yue[\eps]_j}
      \Vol(\holes[\eps])^{\frac{2}{q}}
    \end{aligned}
  \end{equation}
  in which the $\W12(\mfd)$-norms of $\Yue[\eps]_i$, $\Yue[\eps]_j$
  are bounded by Proposition~\ref{prop:norm-bounded}, and
  $\Vol(\holes[\eps])\tendsto0$ by {\hypone}. This finishes
  the proof.
\end{proof}
Propositions~\ref{prop:equation-in-limit}
and~\ref{prop:orthonormality} together establish
Proposition~\ref{prop:main-intermediate} in the case of
$\Vee\in\Cn1(\mfd)$.

To finally recover Proposition~\ref{prop:main-intermediate} fully as
stated for any admissible potential $\Vee$, we use a diagonalisation
argument to approach $\Vee$ along a sequence of potentials in
$\Cn1(\mfd)$.
\begin{prop}[Approximating an admissible potential]
  \label{prop:schroedinger-approx}
  Let $\mfd$ be a Riemannian manifold of dimension $\mfdim\ge2$. Let
  $\Vee$ be an admissible potential on $\mfd$, and let
  $\Bdparam\in\Lp\infty(\bdmfd)$.

  Then there exists a sequence $\sequence{\Vee_n}{n\in\N}$ of
  potentials $\Vee_n\in\Cn1(\mfd)$ converging to $\Vee$ in the space
  of admissible potentials such that as $n\tendsto\infty$, we have
  that
  \begin{itemize}
  \item the Schr\"odinger eigenvalues $\lam{\Sch}(\mfd,\Vee_n,
    \Bdparam)$ tend to
    $\lam{\Sch}(\mfd,\Vee,\Bdparam)$, and
  \item the corresponding eigenfunctions
    $\yuk{\Sch}(\mfd,\Vee_n,\Bdparam)$ tend to $\yuk{\Sch}(\mfd,\Vee,\Bdparam)$ in $\Wz12(\mfd)$.
  \end{itemize}
\end{prop}
\begin{proof}
  We do this by showing that the corresponding Schr\"odinger operators
  converge in the norm-resolvent sense. Let us first present the
  proof in the case $\mfdim\ge3$.

  By density, choose a sequence $\sequence{\Vee_n}{n\in\N}$ in
  $\Cn1(\mfd)$ converging to $\Vee$ in $\Lp{\frac{\mfdim}{2}}(\mfd)$.
  Define now the sesquilinear form
  $\schf:\Wz12(\mfd)\times\Wz12(\mfd)\to\R$ by
  \begin{equation}
    \label{eqn:schf-dfn}
    \schf(u, v)=
    \int_{\mfd}\grad u\cdot \grad v\dv
    +
    \int_{\mfd}uv\Vee\dv
    +
    \int_{\bdmfd}uv\Bdparam\dA
  \end{equation}
  for $u,v\in\Wz12(\mfd)$. To see that $\schf$ is
  well-defined, verify that
  \begin{equation}
    \label{eqn:schf-L1}
    \int_\mfd uv\Vee\dv
    \le
    \norm[\Lp{\frac{\mfdim}{2}}(\mfd)]{\Vee}
    \norm[\Lp{\frac{\mfdim}{\mfdim-2}}(\mfd)]{uv}
    \lesssim_{\mfd}
    \norm[\Lp{\frac{\mfdim}{2}}(\mfd)]{\Vee}
    \norm[\W12(\mfd)]{u}
    \norm[\W12(\mfd)]{v}
  \end{equation}
  for $u,v\in\Wz12(\mfd)$. Let us also define
  $\schf_n$ analogously, with $\Vee_n$ in place of $\Vee$
  in~\eqref{eqn:schf-dfn}.

  By construction, $\schf$ is symmetric; it defines a corresponding
  self-adjoint operator $\schop$ whose domain $\domain{\schop}$ is a form core
  of $\schf$, and which satisfies
  \begin{equation}
    \label{eqn:schop-dfn}
    \schf(u,v)=\idot[\Lp2(\mfd)]{\schop u}{v}
  \end{equation}
  for all $u\in\domain{\schop}$, $v\in\domain{\schf}$. Likewise, each
  $\schf_n$ defines a corresponding operator, which we denote
  $\schop_n$.
  
  We observe the following:
  \begin{itemize}
  \item
    \emph{For $\lambda\in\C\setminus\R$, the resolvent
    $\resolvent{\schop_n}=(\schop_n-\lambda)\inv$ is a continuous
    linear map $\Lp2(\mfd)\to\Wz12(\mfd)$.}
  \item
    \emph{Multiplication by $\Vee_n-\Vee$ defines a continuous linear
    map $\W12(\mfd)\to\dual{\W12(\mfd)}$ of norm no greater that
    $C_\mfd\norm[\Lp{\frac{\mfdim}{2}}(\mfd)]{\Vee_n-\Vee}$.} This follows
    from the computation~\eqref{eqn:schf-L1}.
  \end{itemize}
  Next, we prove that there exists $\lambda\in\C\setminus\R$ such that
  the resolvent $\resolvent{\schop}$ of $\schop$ extends to a continuous
  linear map $\dual{\W12(\mfd)}\to\Wz12(\mfd)$, which we still call
  $\resolvent{\schop}$. The motivation for this is that, by the second
  resolvent identity,
  \begin{equation}
    \label{eqn:resolvent-comp}
    \resolvent{\schop}-\resolvent{\schop_n}
    =
    \resolvent{\schop}(\Vee_n-\Vee)\resolvent{\schop_n}.
  \end{equation}
  To show the norm-resolvent convergence $\schop_n\to\schop$
  is to show that the norm of the left-hand side of
  ~\eqref{eqn:resolvent-comp} vanishes in the limit $\eps\tendsto0$,
  which we do by bounding the norm of each of the three maps composed on
  the right-hand side.

  Towards showing the existence of the extension
  $\resolvent{\schop}:\dual{\W12(\mfd)}\to\Wz12(\mfd)$, see first that
  $\schop-\lambda$ defines a map $\domain{\schop}\to\dual{\W12(\mfd)}$,
  which we still call $\schop-\lambda$, in the following way: for every
  $u\in\domain{\schop}$, define
  $(\schop-\lambda)u\in\dual{\W12(\mfd)}$ by setting
  \begin{equation}
    \label{eqn:schop-lambda-extension}
    \left((\schop-\lambda)u\right)(f)=\idot[\Lp2(\mfd)]{(\schop-\lambda)u}{f}
  \end{equation}
  for every $f\in\W12(\mfd)$. We show in the following that
  $\schop-\lambda$ has a continuous inverse
  $(\schop-\lambda)\inv:\dual{\W12(\mfd)}\to\domain{\schop}$ which,
  composed with the inclusion $\domain{\schop}\subseteq\Wz12(\mfd)$,
  gives us the desired map
  $\resolvent{\schop}:\dual{\W12(\mfd)}\to\Wz12(\mfd)$.
  Let us first briefly recall the following result: for a measure space
  $(X,\mu)$, $p\ge1$, and $f\in\Lp p(X)$, for $\delta>0$, there exists
  $E_\delta\subseteq X$ such that $X$ is essentially bounded on
  $X\setminus E_\delta$, and
  \begin{equation}
    \label{eqn:chebyshev}
    \norm[\Lp p(X)]f
    \le\delta.
  \end{equation}
  Using this, we find that for any $\lambda\in\R$,
  $\delta_1,\delta_2>0$ and for any
  $u\in\domain{\schop}\subseteq\Wz12(\mfd)$, there exist
  $E_{\delta_1}\subseteq\mfd$ and $E_{\delta_2}\subseteq\bd\mfd$,
  such that
  \begin{equation}
    \label{eqn:h-expanding}
    \begin{aligned}
      \norm[\dual{\W12(\mfd)}]{(\schop-\lambda)u}
      \norm[\W12(\mfd)]{u}
      &\ge
      \Re{\idot[\Lp2(\mfd)]{(\schop-\lambda)u}{u}}\\
      &=
      \norm[\Lp2(\mfd)]{\grad u}^2
      +
      \int_\mfd u^2\Vee\dv
      +
      \int_{\bd\mfd}u^2\Bdparam\dA
      -
      \Re{\lambda}\norm[\Lp2(\mfd)]{u}^2\\
      &\ge
      \norm[\Lp2(\mfd)]{\grad u}^2\\
      &\quad-
      \left(
      \norm[\Lp\infty(E_{\delta_1})]{\Vee}\norm[\Lp2(\mfd)]{u}^2
      +{\delta_1}\norm[\Lp{\frac{\mfdim}{2}}(\mfd)]{\Vee}\norm[\W12(\mfd)]{u}^2
      \right)
      \\&\quad-
      \left(
      \norm[\Lp\infty(E_{\delta_2})]{\Bdparam}\norm[\Lp2]{u}^2
      +{\delta_2}\norm[\Lp\infty(\bd\mfd)]{\Bdparam}\norm[\W12(\mfd)]{u}^2
      \right)\\
      &\quad-
      \Re{\lambda}\norm[\Lp2(\mfd)]{u}^2.
    \end{aligned}
  \end{equation}
  We now choose $\delta_1$, $\delta_2$ such that
  \begin{equation}
    0<\delta_1<\frac{1}{3\norm[\Lp{\frac\mfdim2}(\mfd)]{\Vee}}
    \quad\text{and}\quad
    0<\delta_2<\frac{1}{3\norm[\Lp\infty(\bd\mfd)]{\Bdparam}},
  \end{equation}
  and choose $\lambda\in\C$ such that
  \begin{equation}
    \Re{\lambda}<-\big(\norm[\Lp\infty(E_{\delta_1})]{\Vee}+\norm[\Lp\infty(E_{\delta_2})]{\Bdparam}+1\big)
    \quad\text{and}\quad
    \Im{\lambda}\ne0.
  \end{equation}
  From~\eqref{eqn:h-expanding}, we thus deduce
  \begin{equation}
    \norm[\dual{\W12(\mfd)}]{(\schop-\lambda)u}\ge\frac{1}{3}\norm[\W12(\mfd)]{u},
  \end{equation}
  from which it follows that
  $\schop-\lambda:\domain{\schop}\to\dual{\W12(\mfd)}$ is injective
  with closed range. Since $\Lp2(\mfd)$ is dense in $\W12(\mfd)$,
  $\schop-\lambda$ is moreover surjective onto
  $\dual{\W12(\mfd)}$. We conclude that
  $(\schop-\lambda)\inv:\dual{\W12(\mfd)}\to\domain{\schop}$ is
  indeed well-defined and continuous.
  
  Finally, we have for $\lambda\in\C\setminus\R$ that
  \begin{equation}
    \begin{aligned}
      \norm{\resolvent{\schop}-\resolvent{\schop_n}}
      &=
      \norm{\resolvent{\schop}(\Vee_n-\Vee)\resolvent{\schop_n}}\\
      &\le
      \norm{\resolvent{\schop}}
      \norm{\Vee_n-\Vee}
      \norm{\resolvent{\schop_n}}\\
      &\lesssim_{\mfd,\Vee,\lambda}
      \norm[\Lp{\frac{\mfdim}{2}}(\mfd)]{\Vee_n-\Vee}
    \end{aligned}
  \end{equation}
  which vanishes in the limit $\eps\tendsto0$, as desired.

  The case $\mfdim=2$ proceeds identically, with the norms in
  $\Lp{\frac{\mfdim}{2}}(\mfd)$ and
  $\W1{\frac{\mfdim}{\mfdim-1}}(\mfd)$ everywhere replaced by
  $\LlogLone(\mfd)$ and $\W1{2,-\half}(\mfd)$, respectively. Where we
  have used Proposition~\ref{prop:sobolev-multiplication-i} here we
  instead invoke Proposition~\ref{prop:sobolev-multiplication-ii}, and
  where we use the embedding
  $\W1{\frac{\mfdim}{\mfdim-1}}(\mfd)\to\Lp{\frac{2}{\mfdim-2}}(\mfd)$
  we instead apply the embedding~\eqref{eqn:expL-embed}.%
  \end{proof}
This finishes the proof of Proposition~\ref{prop:main-intermediate},
and hence of Proposition~\ref{prop:main}, from which follows also the
claimed results Theorem~\ref{thm:main-nonneg} and
Theorem~\ref{thm:main}.

\section{Flexibility of Schr\"odinger potentials}
\label{sec:flexibility}
In this section we prove Theorem~\ref{thm:flexibility}, which is an
application of our results in the previous sections to the problem of
potential optimisation for the Schr\"odinger eigenvalue problem.

An appropriate setting, described by Buttazzo, Gerolin, Ruffini, and
Velichkov~\cite{bgrv}, in which to formulate optimisation problems for
Schr\"odinger potentials is the class $\Mcap(\mfd)$ of
\emph{capacitary measures} on $\mfd$, which are those Borel measures
which vanish on all sets of capacity zero. For $\mu\in\Mcap(\mfd)$,
one makes sense of the weak formulation of the Schr\"odinger problem
with $\mu$ as the potential by replacing the term $\int_{\mfd}\yuu
v\Vee\dv$ with $\int_{\mfd}\yuu v\,\dee\mu$ in
\eqref{eqn:weak-schroedinger}.

A motivating example for this is the problem of minimising
$\lam{\Sch}(\dom,\Vee)$ for $\Vee:\dom\to\Rnonneg$ in the class
$\mathcal{V}^p_1$ of functions with unit norm in $\Lp p(\dom)$ for a
bounded domain $\dom\subseteq\R^{\mfdim}$ with Dirichlet boundary
conditions. This problem admits no solution in $\mathcal{V}^p_1$,
whereas the same infimum is attained if we enlarge the class of
candidates and consider the same problem for $\Vee$ a non-negative
capacitary potential with unit total mass; see~\cite[Section~3]{bgrv}.

We note that the class of capacitary measures is rather large. A
capacitary measure need not, for example, be absolutely continuous
with respect to the Lebesgue measure.  The measures $\mue$ as defined
in~\eqref{eqn:mues-dfn}, concentrated in a submanifold of co-dimension
one, are capacitary measures on $\mfd$. Thus the following
proposition, whose conclusion is identical to
Theorem~\ref{thm:flexibility} except with measures $\mue$ for the
potentials in place of $\Vee_n$, is already a flexibility result among
capacitary measures for those optimal Schr\"odinger potentials which
are admissible.
\begin{prop}
  \label{prop:flexibility-intermediate}
  Let $\mfd$ be a Riemannian manifold of dimension $\mfdim\ge 2$,
  and let $\Vee$ be an admissible potential on $\mfd$.
  Let $1<p_0<\frac{d}{d-1}$, and let
  \begin{equation}
    \label{eqn:bdconst-choice}
    \bdconst=\eps^{-\frac{\mfdim(p_0-1)}{d-p_0}}.
  \end{equation}
  Let $\zbd\subseteq\bd\mfd$, and let
  $\Bdparam\in\Lp\infty(\bdmfd)$. Let $\dome,\bdparame$ be as
  constructed in Section~\ref{ssec:homogenisation-construction}, and
  let $\mue=\bdparame\area$ as in~\eqref{eqn:mues-dfn}. Then we have
  the following:
  \begin{itemize}
  \item The Schr\"odinger spectrum with potential
    $\mue$ approaches the Schr\"odinger spectrum
    with potential $\Vee$, that is, for each $k\in\Npos$,
    \begin{equation}
      \label{eqn:lam-sch-approx}
      \lam{\Sch}(\mfd,\mue,\Bdparam)\tendsto\lam{\Sch}(\mfd,\Vee,\Bdparam)
    \end{equation}
    as $\eps\tendsto 0$,
  \item For all $p>p_0$, $\mue\tendsto\Vee$ in $\dual{\W1p(\mfd)}$, and
  \item For all $1\le p\le p_0$, There exists a constant $C_{\mfd,\Vee}>0$
    such that
    \begin{equation}
      \label{eqn:lam-sch-unstable}
      \norm[\dual{\W1p(\mfd)}]{\Vee-\mue}>C_{\mfd,\Vee}
    \end{equation}
    for all $\eps>0$.
  \end{itemize}
\end{prop}
\begin{proof}[Proof]
  Noting that
  $\lam{\Sch}(\mfd,\mue,\Bdparam)=\lam{\Rch}(\mfd,\mue,\Bdparam)$, the
  convergence~\eqref{eqn:lam-sch-approx} follows from applying
  Proposition~\ref{prop:main-intermediate}, choosing
  $\dome[\eps]=\mfd$ and $\nue=\mue$. Hypotheses {\hypone} and
  {\hyptwo} are satisfied trivially, and {\hypthree} follows from
  Proposition~\ref{prop:measures-converge}, which also gives us the
  convergence $\mue\tendsto\Vee\volume$ in $\dual{\W1p(\mfd)}$ for
  $p>p_0$.
  
  It therefore remains only to prove~\eqref{eqn:lam-sch-unstable}. We
  do this by explicitly constructing a family of trial functions in
  $\W1p(\mfd)$, and showing that they indeed separate $\mue$ from
  $\Vee$.
  
  Since we are concerned with the limit as $\eps$ tends to $0$, let us
  confine our attention to those $\eps>0$ such that $3\eps$ is no
  greater than the injectivity radius of $\mfd$, so that at each
  $\site\in\sites$ there are geodesic polar co-ordinates $\chart:
  [0,3\eps)\times\sphere[\mfdim-1]\to\mfd$.
    
  For each $\site\in\sites$, define the radial function
  $\overline{\wse}:[0,3\eps)\times\sphere[\mfdim-1]\to\R$
  \begin{equation}
    \overline{\wse}\left(r,\theta\right)=
    \begin{cases}
      \left(1-\frac{1}{\radius}\abs{r-\radius}\right)
      \sign(\Vee(\site))
      &\text{for $\abs{r-\radius}\le\radius$, and}\\
      0&\text{otherwise.}
    \end{cases}
  \end{equation}
  By construction, $\cell\subseteq\ball{\site}{3\eps}$, so the
  composition
  $\wse=\overline{\we}\circ\left(\chart\inv\at{\cell}\right)$ defines
  a function $\wse:\cell\to\R$.

  To bound $(\Vee\dv-\mue)(\we)$ away from zero, we show that
  $\mue(\we)$ is bounded away from zero as $\eps\tendsto0$,
  whereas $(\Vee\dv)(\we)=\int_{\mfd}\we\Vee\dv$ vanishes in the
  limit.
  
  The function $\wse$ is supported on a neighbourhood of $\bd\hole$
  contained within $\ball{\site}{2\radius}$, which for all
  sufficiently small $\eps>0$ is in turn contained entirely within
  $\cell$. For all such $\eps$, we finally define $\we:\mfd\to\R$
  piecewise by setting $\we\at{\cell}=\wse$ for each
  $\site\in\sites$. The function $\we$ is well-defined since $\wse$
  agree on the boundary between adjacent cells, where they vanish;
  and we have $\we\in\Cn{}(\mfd)$ by construction.
  
  We have the following estimate for the volume of the set on
  which $\we$ is supported:
  \begin{equation}
    \label{eqn:int-vol-supp-we}
    \Vol(\supp \we)
    \le
    \underbrace{\card{\sites}}_{\asymp_{\mfd,\Vee}\eps^{-\mfdim}}
    \underbrace{\Vol(\ball{\site}{2\radius})}_{\lesssim_{\mfd,\Vee}\bdconst^{-\frac{\mfdim}{\mfdim-1}}\eps^{\frac{\mfdim^2}{\mfdim-1}}}
    \lesssim_{\mfd,\Vee}
    \bdconst^{-\frac{\mfdim}{\mfdim-1}}\eps^{\frac{\mfdim}{\mfdim-1}}
  \end{equation}
  recalling~\eqref{eqn:bdconst-hyp}, we see that this vanishes in
  the limit $\eps\tendsto0$. Since $\Vee\in\Lp1$, we have that
  $\Vee\volume$ is absolutely continuous with respect to $\volume$;
  consequently
  \begin{equation}
    \label{eqn:int-vee}
    \bigabs{\int_\mfd\we\Vee\dv}
    \le
    \int_{\supp\we}\abs{\Vee}\dv\tendsto0.
  \end{equation}
  
  On the other hand, to estimate $\mue(\we)$, we use the fact that the
  function $\wse$ takes value $\pm1$ on $\bd\holes\subseteq\mfd$,
  taking on each connected component the same sign as the boundary
  parameter there. Thus
  \begin{equation}
    \label{eqn:int-mease}
    \mue(\we)=
    \int_{\bd\holes}\abs{\bdparame}\dA
    \asymp_{\mfd,\Vee}1.
  \end{equation}
  The asymptotic estimate comes from
  $\lim_{\eps\tendsto0}\mue(\we)=\norm[\Lp1(\mfd)]{\Vee}$, by
  Corollary~\ref{cor:measures-weak-converge}.

  To conclude from this that the measures $\Vee$ and $\mue$ are
  separated in $\dual{\W1p(\mfd)}$, it remains to verify that the
  trial functions $\we$ are bounded in $\W1p(\mfd)$ for
  small~$\eps>0$. Using some of the same estimates
  from~\eqref{eqn:int-vol-supp-we}, the $\Lp p(\mfd)$-norm of $\we$ is
  bounded by
  \begin{equation}
    \label{eqn:Lp1-we}
    \norm[\Lp p(\mfd)]{\we}^p
    =
    \sum_{\site\in\sites}\int_{\cell}\abs{\wse}\dv
    \le
    \card{\sites}\Vol(\ball{\site}{2\radius})
    \underbrace{\norm[\Lp\infty(\cell)]{\wse}}_{=1}
    \lesssim_{\mfd,\Vee}
    \eps^{\frac{\mfdim}{\mfdim-1}}.
  \end{equation}
  To control the $\Lp p(\mfd)$-norm of $\grad\we$, first compute in
  co-ordinates
  \begin{equation}
    \label{eqn:we-bounded-cell}
    \begin{aligned}
      \int_{\cell}\abs{\grad \wse}^p\dv
      &\asymp_{\mfd,p}
      \int_{0}^{2\radius}
      \int_{\sphere[\mfdim-1]}
      \abs{\grad[]\overline{\wse}}^p
      \,\dee\theta
      \,\dee r
      \\
      &\lesssim_{\mfd,\Vee,p}
      \Vol[\mfdim-1](\ball[]{0}{2\radius})
      \radius^{-(p-1)}\\
      &\lesssim_{\mfd,\Vee,p}
      \bdconst^{-\frac{\mfdim-p}{\mfdim-1}}
      \eps^{\frac{\mfdim(\mfdim-p)}{\mfdim-1}}.
    \end{aligned}
  \end{equation}
  Aggregating over all cells, we then have that
  \begin{equation}
    \label{eqn:Lp1-grad-we}
    \norm[\Lp p(\mfd)]{\grad\we}^p
    =
    \sum_{\site\in\sites}\int_{\cell}\abs{\grad\wse}^p\dv
    \lesssim_{\mfd,\Vee,p}
    \bdconst^{-\frac{\mfdim-p}{\mfdim-1}}\eps^{-\frac{\mfdim(p-1)}{\mfdim-1}}
    \lesssim_{\mfd,\Vee,p}
    \eps^{\frac{-\mfdim(p-p_0)}{\mfdim-p_0}}
  \end{equation}
  by substituting in~\eqref{eqn:bdconst-choice}. This stays bounded as
  soon as $p\le p_0$, and tends to zero as for each $p<p_0$ as
  $\eps\tendsto0$; from this follows boundedness uniform in $p$, that
  is,
  \begin{equation}
    \norm[\Lp p(\mfd)]{\grad\we}^p\le C_{\mfd,\Vee}.
  \end{equation}
  This finishes the proof.
\end{proof}
To finally recover Theorem~\ref{thm:flexibility} from
Proposition~\ref{prop:flexibility-intermediate}, we show that the
potential $\mue$ can be in turn approached in $\dual{\W1p(\mfd)}$ by
potentials which have smooth functions as their densities. We
formulate this proposition in slightly greater generality than is
immediately necessary, showing in general that it is possible to
approximate in this way a measure supported on a smooth submanifold of
co-dimension one. The proof proceeds by explicitly constructing smooth
potentials supported on successively smaller tubular neighbourhoods of
the submanifold, following an argument that is essentially identical
to that of \cite[Theorem~5.2]{gkl}.
\begin{prop}[Approximating a potential concentrated in a submanifold]
  \label{prop:submfd-approx}
  Let $1<p<\infty$. Let $\mfd$ be a Riemannian manifold of dimension
  $\mfdim$, and let $\submfd\subseteq\interior{\mfd}$ be a smooth
  embedded submanifold of co-dimension one. Let
  $\bdparam\in\Cn\infty(\submfd)$, and write $\dA$ for the
  $(\mfdim-1)$-dimensional Hausdorff measure on $\submfd$. Then there
  exist a family of functions $\sequence{\Veee}{\eps>0}$ in
  $\Cn\infty(\mfd)$ such that $\Veee\dv\tendsto\bdparam\dA$ in
  $\dual{\W1p(\mfd)}$ as $\eps\tendsto0$.
\end{prop}
\begin{proof}
  Let us write $\pi:\Nb\submfd\to\submfd$ for
  the normal bundle to $\submfd$ in $\mfd$. By the tubular
  neighbourhood theorem and the compactness of $\cl{\mfd}$, there
  exists $\eps_0>0$ such that $\Nchart:\Nb\submfd\to\mfd$ defined~by
  \begin{equation}
    \Nchart(x,w)=\exp^\mfd_xw
  \end{equation}
  for $x\in\submfd$, $w\in\Np x\submfd=\pi\inv(x)$ is a
  diffeomorphism from an $\eps_0$-neighbourhood of the zero section
  of $\Nb\submfd$ onto a neighbourhood of $\submfd$ in $\mfd$.

  For $0<\eps<\eps_0$, let us write
  $\Nbnhd=\{(x,w)\in\Nb\submfd{\,:}\abs{w}<\eps\}$ and
  $\Npnhd=\Npnhd\cap\Np x\submfd$. Then for any bounded
  $f:\mfd\to\R$ supported within $\Nchart(\Nbnhd)$, we have the
  estimate
  \begin{equation}
    \label{eqn:Nb-integral-error}
    \int_{\mfd}f\dv
    =
    \left(1+\bigO[\mfd](\eps)\right)
    \int_{\submfd}
    \int_{\Npnhd}
    (f\circ\Nchart)(x,w)
    \,\dee w
    \,\dee x
  \end{equation}
  in which the integration against the variables $x$ and $w$ are
  with respect to the $(\mfdim-1)$-dimensional Hausdorff measure
  $\area$ on $\submfd$, and with respect to the measure on $\Np
  x\submfd$ induced by the Riemannian metric on $\Tp x\mfd$, respectively.

  Now fix a function $\rho\in\Cn\infty(\R)$ which is non-negative,
  supported in $[-1,1]$, and is such that $\int_{\R}\rho(t)\,\dee
  t=1$. Let us write $\rho_\eps(t)$ for $\eps\inv\rho(\eps t)$. For
  $0<\eps<\eps_0$, we then define $\Veee\in\Cn\infty(\mfd)$ by
  \begin{equation}
    \Veee(y)=\begin{cases}
    \rho_\eps(\abs{w})\bdparam(x)&\text{for $y=\Nchart(x,w)$ with $(x,w)\in\Nbnhd$,}\\
    0&\text{otherwise.}
    \end{cases}
  \end{equation}
  For $0<\eps<\eps_0$, we then have by~\eqref{eqn:Nb-integral-error}
  that
  \begin{equation}
    \label{eqn:Nb-integral-calc}
    \begin{aligned}
      &\int_\submfd v\bdparam\dA
      -
      \int_\mfd v\Veee\dv
      \\
      &\quad=
      \int_\submfd v(x)\bdparam(x)\,\dee x
      -
      (1+\bigO[\mfd](\eps))
      \int_\submfd\int_{\Npnhd}(v\circ\Nchart)(x,w)\rho_\eps(\abs{w})\bdparam(x)\,\dee w\,\dee x\\
      &\quad=
      \left(1+\bigO[\mfd](\eps)\right)
      \int_\submfd
      \bdparam(x)
      \left(
      v(x)
      -
      \int_{\Npnhd}
      (v\circ\Nchart)(x,w)
      \rho_\eps(\abs{w})
      \,\dee w
      \right)
      \,\dee x\\
      &\quad=
      \left(1+\bigO[\mfd](\eps)\right)
      \int_\submfd
      \bdparam(x)
      \left(
      v(x)
      -
      v(x)
      \int_{\Npnhd}
      \rho_\eps(\abs{w})
      \,\dee w
      +
      \int_{\Npnhd}
      v(x)
      \rho_\eps(\abs{w})
      \,\dee w
      -
      \int_{\Npnhd}
      (v\circ\Nchart)(x,w)
      \rho_\eps(\abs{w})
      \,\dee w
      \right)
      \,\dee x\\
      &\quad=
      \left(1+\bigO[\mfd](\eps)\right)
      \int_\submfd
      \bdparam(x)
      \left(
      v(x)
      \Big(
      1-\int_{\Npnhd}\rho_\eps(\abs{w})\,\dee w
      \Big)
      +
      \int_{\Npnhd}
      \left(v(x)-(v\circ\Nchart)(x,w)\right)
      \rho_\eps(\abs{w})
      \,\dee w
      \right)
      \,\dee x.
    \end{aligned}
  \end{equation}
  By construction of $\rho_\eps$, we have that
  $\int_{\Npnhd}\rho_\eps(\abs{w})\,\dee w=1$. For $w\in\Npnhd$,
  \begin{equation}
    \bigabs{v(x)-v\circ\Nchart(x,w)}
    =
    \bigabs{(v\circ\Nchart)(x,0)-(v\circ\Nchart)(x,w)}
    \le
    \int_0^1
    \abs{\pd{t}(v\circ\Nchart)(x,tw)}
    \,\dee t
    \le
    \int_{\Npnhd}
    \bigabs{\left(\grad v\at{\Nchart(x,w)}\right)}
    \,\dee w
  \end{equation}
  so for the surviving integral on the final line of
  \eqref{eqn:Nb-integral-calc}, we have that
  \begin{equation}
    \label{eqn:Nb-integral-calc-2}
    \begin{aligned}
      \bigabs{
        \int_{\submfd}\bdparam(x)
        \int_{\Npnhd}
        \left(v(x)-(v\circ\Nchart)(x,w)\right)
        \rho_\eps(\abs{w})
        \,\dee w
      }
      &\le
      \norm[\Lp\infty(\submfd)]{\bdparam}
      \norm[\Lp\infty(\Npnhd)]{\rho_\eps}
      \int_{\submfd}
      \int_{\Npnhd}
      \int_{\Npnhd}
      \bigabs{\left(\grad v\at{\Nchart(x,w_1)}\right)}
      \,\dee w_1
      \,\dee w_2
      \,\dee x\\
      &=
      \norm[\Lp\infty(\submfd)]{\bdparam}
      \norm[\Lp\infty(\Npnhd)]{\rho_\eps}
      \int_{\Npnhd}
      \underbrace{
        \int_{\submfd}
        \int_{\Npnhd}
        \bigabs{\left(\grad v\at{\Nchart(x,w_1)}\right)}
        \,\dee w_1
        \,\dee x
      }_{=(1+\bigO[\mfd](\eps))\norm[\Lp1(\Nchart(\Nbnhd))]{\grad v}}
      \,\dee w_2\\
      &\lesssim_{\mfd,\bdparam}
      \norm[\Lp1(\Nchart(\Nbnhd))]{\grad v}
    \end{aligned}
  \end{equation}
  as $\eps\tendsto0$, where on the final line we have used the fact
  that $\Npnhd$ has volume $2\eps$, and that
  $\norm[\Lp\infty(\Npnhd)]{\rho_\eps}=\eps\inv\norm[\Lp\infty(\R)]{\rho}$.

  Finally, by H\"older's inequality, for $p>1$ we have that
  \begin{equation}
    \label{eqn:Nb-integral-calc-3}
    \norm[\Lp1(\Nchart(\Nbnhd))]{\grad v}
    \le
    \underbrace{
      \norm[\Lp{\frac{p}{p-1}}(\mfd)]{\chf{\Nchart(\Nbnhd)}}
    }_{=\bigO[\submfd](\eps^{\frac{p-1}{p}})}
    \underbrace{
      \norm[\Lp p(\mfd)]{\grad v}
    }_{\le\norm[\W1p(\mfd)]{v}}.
  \end{equation}

  Substituting~\eqref{eqn:Nb-integral-calc-2}
  and~\eqref{eqn:Nb-integral-calc-3} into
  \eqref{eqn:Nb-integral-calc}, we have altogether that
  \begin{equation}
    \bigabs{
      \int_\submfd v\bdparam\dA
      -
      \int_\mfd v\Veee\dv
    }
    \lesssim_{\mfd,\bdparam}
    \eps^{\frac{p-1}{p}}\norm[\W1p(\mfd)]{v}.
  \end{equation}
  As $\eps\tendsto0$, since $\eps^{\frac{p-1}{p}}$ vanishes in the
  limit, we conclude that $\Veee\dv\tendsto\bdparam\dA$ in
  $\dual{\W1p(\mfd)}$, as desired.
\end{proof}
At last, we arrive at
\begin{proof}[Proof of Theorem~\ref{thm:flexibility}]    
  By a diagonalisation argument it suffices to prove, for
  $1<p<\infty$, that each of the measures $\mue$ in
  Theorem~\ref{thm:flexibility} can be approached by a family
  $\sequence{\Veee}{\eps>0}$ of smooth functions such that
  $\Vee\volume\tendsto\mue$ in $\dual{\W1p(\mfd)}$. Invoking
  Proposition~\ref{prop:submfd-approx} with $\submfd=\bd\holes$ and
  $\bdparam=\bdparame$ finishes the proof.
\end{proof}

\section*{Acknowledgements}
This work was undertaken during the author's PhD studies at King's
College London under the supervision of Jean Lagac\'e. The author also
thanks Alexandre Girouard for many helpful comments on a draft of this
article.
\bibliography{main}

\providecommand{\bysame}{\leavevmode\hbox to3em{\hrulefill}\thinspace}
\providecommand{\MR}{\relax\ifhmode\unskip\space\fi MR }
\providecommand{\MRhref}[2]{%
  \href{http://www.ams.org/mathscinet-getitem?mr=#1}{#2}
}
\providecommand{\href}[2]{#2}
\begin{thebibliography}{10}

\bibitem{anne-post}
Colette Ann\'e and Olaf Post, \emph{Wildly perturbed manifolds: norm resolvent
  and spectral convergence}, Journal of Spectral Theory \textbf{11} (2021),
  229--279.

\bibitem{b-s}
Colin Bennett and Robert Sharpley, \emph{Interpolation of operators}, Pure and
  Applied Mathematics, vol. 129, Academic Press, 1988.

\bibitem{b-p}
Lorenzo Brasco and Guido De~Philippis, \emph{Spectral inequalities in
  quantitative form}, Shape Optimization and Spectral Theory (Antoine Henrot,
  ed.), De Gruyter Open Ltd, Warsaw/Berlin, 2017, pp.~201--281.

\bibitem{bucur-nahon}
Dorin Bucur and Mickaël Nahon, \emph{Stability and instability issues of the
  {Weinstock} inequality}, Transactions of the American Mathematical Society
  \textbf{374} (2020), 1.

\bibitem{bgrv}
Giuseppe Buttazzo, Augusto Gerolin, Berardo Ruffini, and Bozhidar Velichkov,
  \emph{Optimal potentials for {Schrodinger} operators}, Journal de l'Ecole
  Polytechnique - Mathematiques \textbf{1} (2013), 71--100.

\bibitem{cdr}
Kirill Cherednichenko, Patrick Dondl, and Frank R\"osler, \emph{Norm-resolvent
  convergence in perforated domains}, Asymptotic Analysis \textbf{110} (2018),
  no.~3--4, 163--184.

\bibitem{cianchi}
Andrea Cianchi, \emph{A sharp embedding theorem for {Orlicz--Sobolev} spaces},
  Indiana University Mathematics Journal \textbf{45(1)} (1996), 39--65.

\bibitem{cioranescu-murat}
Doina Cioranescu and Fran{\c{c}}ois Murat, \emph{A strange term coming from
  nowhere}, Topics in the Mathematical Modelling of Composite Materials (Andrej
  Cherkaev and Robert Kohn, eds.), Birkh\"auser Boston, Boston, MA, 1997,
  pp.~45--93.

\bibitem{fraser-schoen}
Ailana Fraser and Richard Schoen, \emph{The first {Steklov} eigenvalue,
  conformal geometry, and minimal surfaces}, Advances in Mathematics
  \textbf{226} (2011), no.~5, 4011--4030.

\bibitem{gkl}
Alexandre Girouard, Mikhail Karpukhin, and Jean Lagac\'e, \emph{Continuity of
  eigenvalues and shape optimisation for {Laplace} and {Steklov} problems},
  Geometric and Functional Analysis \textbf{31} (2021), 513--561.

\bibitem{girouard-lagace}
Alexandre Girouard and Jean Lagac\'e, \emph{Large {Steklov} eigenvalues via
  homegenisation on manifolds}, Inventiones Mathematicae \textbf{226} (2021),
  1011--1056.

\bibitem{hkp}
Asma Hassannezhad, Gerasim Kokarev, and Iosif Polterovich, \emph{Eigenvalue
  inequalities on {Riemannian} manifolds with a lower {Ricci} curvature bound},
  Journal of Spectral Theory \textbf{6} (2016), 807--835.

\bibitem{henrot}
Antoine Henrot, \emph{Extremum problems for eigenvalues of elliptic operators},
  Springer Science \& Business Media, 2006.

\bibitem{kaizu}
Satoshi Kaizu, \emph{The {Robin} problems on domains with many tiny holes},
  Proceedings of the Japan Academy, Series A, Mathematical Sciences \textbf{61}
  (1985), no.~2, 39--42.

\bibitem{karpukhin-lagace}
Mikhail Karpukhin and Jean Lagac\'e, \emph{Flexibility of {Steklov} eigenvalues
  via boundary homogenisation}, Annales math\'ematiques du Qu\'ebec \textbf{48}
  (2022), 175--186.

\bibitem{knps}
Mikhail Karpukhin, Micka\"el Nahon, Iosif Polterovich, and Daniel Stern,
  \emph{Stability of isoperimetric inequalities for {Laplace} eigenvalues on
  surfaces}, 2021, arXiv:2106.15043.

\bibitem{kokarev}
Gerasim Kokarev, \emph{Variational aspects of {Laplace} eigenvalues on
  {Riemannian} surfaces}, Advances in Mathematics \textbf{258} (2014),
  191--239.

\bibitem{mazya}
Vladimir~G. Maz'ya and Tatyana~O. Shaposhnikova, \emph{Theory of {Sobolev}
  multipliers with applications to differential and integral operators},
  Springer, 2009.

\end{thebibliography}
\bibliographystyle{amsplain}
\end{document}